\documentclass[11pt]{amsart} \textwidth=14.5cm \oddsidemargin=1cm
\usepackage{amsmath,wasysym}
\usepackage{amsxtra}
\usepackage{amscd}
\usepackage{amsthm}
\usepackage{amsfonts}
\usepackage{amssymb}
\usepackage{eucal}
\usepackage{graphics, color}
\usepackage{hyperref,mathrsfs}
\usepackage[usenames,dvipsnames]{xcolor}
\usepackage{tikz}
\usepackage{stmaryrd}
\usepackage[all]{xy}
\usepackage{hyperref}
\usepackage{color}
\usepackage{bbm} 
\allowdisplaybreaks

\hypersetup{colorlinks,linkcolor=blue, urlcolor=blue,
citecolor=blue}

\newtheorem{thm}{Theorem} [section]
\newtheorem{lem}[thm]{Lemma}

\newtheorem{prop}[thm]{Proposition}
\newtheorem{conj}[thm]{Conjecture}

\theoremstyle{definition}

\newtheorem{example}[thm]{Example}

\newtheorem{rem}[thm]{Remark}

\numberwithin{equation}{section}

\newcommand{\A}{\mathcal A}

\newcommand{\D}{\mathcal{D}}

\newcommand{\End}{{\mathrm{End}}}

\newcommand{\HH}{\mathcal H}
\newcommand{\Hom}{{\mathrm{Hom}}}

\newcommand{\Z}{{\mathbb Z}}

\newcommand{\nc}{\newcommand}
\nc{\browntext}[1]{\textcolor{brown}{#1}}
\nc{\greentext}[1]{\textcolor{green}{#1}}
\nc{\redtext}[1]{\textcolor{red}{#1}}
\nc{\bluetext}[1]{\textcolor{blue}{#1}}
\nc{\brown}[1]{\browntext{ #1}}
\nc{\green}[1]{\greentext{ #1}}
\nc{\red}[1]{\redtext{ #1}}
\nc{\blue}[1]{\bluetext{ #1}}

\title[Cells in modified $\imath$quantum groups of type AIII and related Schur algebras]
{Cells in modified $\imath$quantum groups of type AIII and related Schur algebras}

\author[Weideng Cui]{Weideng Cui}
\address{School of Mathematics, Shandong University, Jinan, Shandong 250100, China}
\email{cwdeng@amss.ac.cn}


\begin{document}

\begin{abstract}
We provide a combinatorial characterization of two-sided cells in modified $\imath$quantum groups of type AIII. Our approach is to lift a corresponding description of two-sided cells in $\jmath$-Schur algebras associated to Iwahori--Hecke algebras of type $B$. We further give a combinatorial description of two-sided cells in other two kinds of Schur-type algebras, namely $\imath$-Schur algebras and $\tilde{\imath}$-Schur algebras.
\end{abstract}

\maketitle
\setcounter{tocdepth}{1}
\tableofcontents

\section{Introduction}
\subsection{}
In the study of primitive ideals in the enveloping algebra of a complex semisimple Lie algebra, Joseph \cite{Jo77} has defined the concept of left, right and two-sided cells in the Weyl group $W$. Joseph showed that each left cell can naturally carry a representation of $W$, and the definition of left cells and the corresponding Weyl group representations involved some unknown quantities: the multiplicities in the Jordan--H\"{o}lder series of the Verma modules with highest weight. Kazhdan and Lusztig \cite{KL79} have proposed a conjectural algorithm for these multiplicities, which was soon proved by Beilinson--Bernstein \cite{BB81} and Brylinski--Kashiwara \cite{BK81}, respectively; moreover, they defined left, right and two-sided cells for an arbitrary Coxeter group through the so-called Kazhdan--Lusztig basis. Kazhdan and Lusztig's definition of cells has the advantage that it gives rise not only to representations of the Coxeter group, but also of the corresponding Iwahori--Hecke algebra. Later on, Barbasch and Vogan \cite{BV82, BV83} determined explicitly the equivalence relation $E \sim_{LR} E'$ on irreducible representations of $W$ defined by the requirement that $E, E'$ appear in the same two-sided cell.

In \cite[\S29.4.2]{Lu10}, Lusztig gave the definition of left, right and two-sided cells for an associative algebra with a given basis; in \cite[Part IV]{Lu10} and also \cite{Lu95}, Lusztig has studied the structure of cells in the modified quantum group with respect to the canonical basis. In \cite{Du96}, Du has given a combinatorial description of two-sided cells in the $q$-Schur algebra of type $A$ with respect to its Kazhdan--Lusztig type basis, in terms of which he also gave a classification of two-sided cells in the modified quantum group of type $A$. In \cite{McG03}, McGerty has obtained an explicit description of two-sided cells in the modified quantum group of affine type $A$ via first investigating the structure of cells in affine $q$-Schur algebras of type $A$.

\subsection{}
Given an involution $\theta$ on a complex simple Lie algebra $\mathfrak{g}$, we can obtain a symmetric pair $(\mathfrak{g}, \mathfrak{g}^{\theta})$, or a pair of enveloping algebras $(\mathbf{U}(\mathfrak{g}), \mathbf{U}(\mathfrak{g}^{\theta}))$, where $\mathfrak{g}^{\theta}$ denotes the fixed point subalgebra of $\mathfrak{g}$ under $\theta$. The classification of symmetric pairs of finite type is closely related to that of real simple Lie algebras, which can be described in terms of the Satake diagrams. In the 1990's, Noumi, Sugitani and Dijkhuizen \cite{NS95, N96, NDS97}, based on solutions to the reflection equations, constructed quantum group analogs of $\mathbf{U}(\mathfrak{g}^{\theta})$ as coideal subalgebras of $\mathbf{U}_{q}(\mathfrak{g})$ for all $\mathfrak{g}$ of classical type, where $\mathbf{U}_{q}(\mathfrak{g})$ is the quantized enveloping algebra of $\mathfrak{g}$. Independently, Letzter \cite{Le99, Le02} systematically developed the theory of quantum symmetric pairs $(\mathbf{U}_{q}(\mathfrak{g}), \mathbf{U}^{\imath})$ of finite type as a quantization of $(\mathbf{U}(\mathfrak{g}), \mathbf{U}(\mathfrak{g}^{\theta}))$. The algebra $\mathbf{U}^{\imath}$ will be referred to as an $\imath$quantum group. Kolb \cite{Ko14} has further studied and generalized Letzter's theory to the Kac--Moody type.

In recent years, Bao and Wang \cite{BW18} have initiated a new theory of the canonical basis (called $\imath$-canonical basis) arising from quantum symmetric pairs of type AIII without black nodes; as an application, they established for the first time a Kazhdan--Lusztig theory for the BGG category $\mathcal{O}$ of the ortho-symplectic Lie superalgebras $\mathfrak{o}\mathfrak{s}\mathfrak{p}(2m+1|2n)$ (see also \cite{B17} for the case of $\mathfrak{o}\mathfrak{s}\mathfrak{p}(2m|2n)$). Simultaneously and independently from \cite{BW18}, Ehrig and Stroppel \cite{ES18} have discovered the connections between coideal subalgebras of type AIII and the parabolic category $\mathcal{O}$ of type $D$; in particular, they gave a categorification of the coideal subalgebras and (quantum) skew Howe duality. Subsequently, Bao and Wang have developed a general theory of $\imath$-canonical bases for modified $\imath$quantum groups arising from quantum symmetric pairs of arbitrary finite type in \cite{BW18b} and of Kac--Moody type in \cite{BW21}, respectively.

In \cite{BKLW18} the authors defined a new Schur-type algebra $S^{\jmath}(n, d)$, which is attached to the Iwahori--Hecke algebra $\HH$ of type $B_d$ with equal parameters and called the $\jmath$-Schur algebra, together with a canonical basis. Via a stabilization procedure as $d$ varies, these algebras give rise to a limit algebra which was shown to be isomorphic to the idempotented $\imath$quantum group $\dot{\mathbf{U}}^{\jmath}(\mathfrak{g}\mathfrak{l}_{n})$. From the process, they established a natural surjective algebra homomorphism from $\dot{\mathbf{U}}^{\jmath}(\mathfrak{g}\mathfrak{l}_{n})$ to $S^{\jmath}(n, d)$. Based on these, Li and Wang \cite{LiW18} further defined and studied the modified $\imath$quantum group $\dot{\mathbf{U}}^{\jmath}(\mathfrak{s}\mathfrak{l}_{n})$ as well as its $\jmath$-canonical basis (\cite[Theorem 5.5]{LiW18}), which has many remarkable properties. For example, the $\jmath$-canonical basis has positive structure constants, and it is almost orthonormal and admits positivity with respect to a bilinear form on $\dot{\mathbf{U}}^{\jmath}(\mathfrak{s}\mathfrak{l}_{n})$. Besides, there also exists a natural surjective homomorphism from $\dot{\mathbf{U}}^{\jmath}(\mathfrak{s}\mathfrak{l}_{n})$ to $S^{\jmath}(n, d)$, which sends canonical basis elements to canonical basis elements or zero as proven in \cite{BSWW18}. Later on, Lai and Luo \cite{LL21} have studied the $\jmath$-Schur algebras associated to Iwahori--Hecke algebras of type $B/ C$ with unequal parameters. Moreover, the authors \cite{FLLLW20, FLLLW22} have given two constructions of the modified $\imath$quantum groups associated to some coideal subalgebras of quantum groups of affine type $A$ via a geometric and algebraic approach, respectively.

In recent years, the $\imath$quantum groups have been studied through various approaches (see \cite{W21} for a survey).

\subsection{}
In this paper, we would like to give a characterization of two-sided cells in the modified $\imath$quantum group $\dot{\mathbf{U}}^{\jmath}(\mathfrak{s}\mathfrak{l}_{n})$ with respect to the $\jmath$-canonical basis defined by Li and Wang. In order to do this, we first give a combinatorial description of two-sided cells in $S^{\jmath}(n, d)$ with respect to its canonical basis, and then lift this description to $\dot{\mathbf{U}}^{\jmath}(\mathfrak{s}\mathfrak{l}_{n})$ via the algebra homomorphism mentioned above.


This paper has provided the first step towards understanding the structure of cells in modified $\imath$quantum groups. Since the $\imath$-canonical basis exists for modified $\imath$quantum groups of general type by Bao and Wang's work, we are planning to investigate the corresponding structure of cells in a future publication.

The paper is organized as follows. In Section 2, we recall the definition of cells for an arbitrary associative algebra with a given basis following Lusztig.

In Section 3, we first present a classification of left, right and two-sided cells in $S^{\jmath}(n, d)$ in terms of the ones in $\HH$. Then we provide a combinatorial description of two-sided cells in $S^{\jmath}(n, d)$, using that of two-sided cells in $\HH$ due to Barbasch and Vogan and also a result of Du. Finally, we state a conjecture on the number of left cells in a two-sided cell of $S^{\jmath}(n, d)$.

In Section 4, we lift the combinatorial description of two-sided cells in $S^{\jmath}(n, d)$ to give a characterization of two-sided cells in the modified $\imath$quantum group $\dot{\mathbf{U}}^{\jmath}(\mathfrak{s}\mathfrak{l}_{n})$.

In Section 5, we focus on another Schur-type algebra associated to $\HH$, namely the $\imath$-Schur algebra $S^{\imath}(n, d)$. Parallel to Section 3 we formulate a combinatorial description of two-sided cells in $S^{\imath}(n, d)$.

In Section 6, we consider the $\tilde{\imath}$-Schur algebra $\tilde{S}^{\imath}(n, d)$ attached to $\mathcal{H}_{C_d}^{1}$, where $\mathcal{H}_{C_d}^{1}$ is the specialization at $p = 1$ of the Iwahori--Hecke algebra $\mathcal{H}_{C_d}^{p}$ of type $C_d$ with unequal parameters $p$ and $q$. We give an approach to determining whether or not two canonical basis elements of $\tilde{S}^{\imath}(n, d)$ lie in the same two-sided cell.




\vspace{2mm}
\noindent {\bf Acknowledgement.}
The work was initiated during my visit to University of Virginia in 2018-2019 under the guidance of Professor Weiqiang Wang. I would like to express my great gratitude to him for introducing me to the interesting problem, for insightful discussions and for helpful comments on my drafts. For example, the conjecture in Remark \ref{remark-tilde-Schur} and the statements in Remark \ref{affine-type-based module maps} are all proposed by him. I thank Professors Yiqiang Li and C. Stroppel for many helpful comments after posting the preprint on arXiv. I also thank an anonymous referee for many helpful suggestions. I thank University of Virginia for hospitality and support. I also acknowledge the support from China Scholarship Council. The author is partially supported by Young Scholars Program of Shandong University, Shandong Provincial Natural Science Foundation (Grant No. ZR2021MA022) and the NSF of China (Grant No. 11601273).


\section{Definition of cells}
Let $q$ be an indeterminate and set $\A=\Z[q,q^{-1}]$. Assume that $R$ is an arbitrary ring. Let $\mathfrak{A}$ be an associative algebra over $R$ and $\mathbb{B}$ a basis of $\mathfrak{A}$
as an $R$-module. We do not assume that $\mathfrak{A}$ has $1$. The structure constants $h_{a, b}^{c}\in R$ of $\mathfrak{A}$ (where $a, b, c\in \mathbb{B}$) are defined by $ab=\sum\limits_{c\in \mathbb{B}}h_{a, b}^{c}c$. Then $\mathbb{B}$ is divided into {\em cells}, by Lusztig \cite[\S29.4.2]{Lu10}, via the equivalence relations on $\mathbb{B}$ defined as follows.

For $b, b'\in \mathbb{B}$, we write $b'\leftarrow_{L} b$ (resp. $b'\leftarrow_{R} b$) if there exists an element $c\in \mathbb{B}$ such that the coefficient of $b'$ is nonzero when expanding $cb$ (resp. $bc$); we say that $b'\preceq_{L} b$ (resp. $b'\preceq_{R} b$) if there exists a sequence $b_1=b, b_2,\ldots, b_n=b'$ in $\mathbb{B}$ such that $b_{i+1}\leftarrow_{L} b_i$ (resp. $b_{i+1}\leftarrow_{R} b_i$) for any $i=1,\ldots, n-1$. We say that $b'\preceq_{LR} b$ if there is a sequence $b_1=b, b_2,\ldots, b_n=b'$ in $\mathbb{B}$ such that for any $i\in \{1,\ldots, n-1\}$ we have either $b_{i+1}\leftarrow_{L} b_i$ or $b_{i+1}\leftarrow_{R} b_i$. Clearly $\preceq_{L}$, $\preceq_{R}$, $\preceq_{LR}$ are preorders on $\mathbb{B}$. For $\star\in \{L, R, LR\}$ and $b, b'\in \mathbb{B}$, we say that $b\sim_{\star}b'$ if $b\preceq_{\star} b'$ and $b'\preceq_{\star} b$. Then $\sim_{L}$, $\sim_{R}$, $\sim_{LR}$ give rise to equivalence relations on $\mathbb{B}$; the corresponding equivalence classes are called {\em left, right} and {\em two-sided cells} of $\mathfrak{A},$ respectively.

The following lemma provides another characterization of these preorders.

\begin{lem}\label{another-charact-cells}
Assume that $R=\A$ and all $h_{a, b}^{c}\in \mathbb{N}[q, q^{-1}]$. Then $b'\preceq_{L} b''$ (resp. $b'\preceq_{R} b''$; $b'\preceq_{LR} b''$) if and only if there exists $\beta$ (resp. $\beta'$; $\beta, \beta'$) in $\mathbb{B}$ such that $e_{b'}\neq 0$ (resp. $f_{b'}\neq 0$; $g_{b'}\neq 0$), where $e_{b'}$ (resp. $f_{b'}$; $g_{b'}$) is defined by
\begin{align*}
\beta b''=\sum_{c\in \mathbb{B}}e_{c}c, \quad e_{c}\in \A,
\end{align*}
\begin{align*}
(resp.~ b''\beta'=\sum_{c\in \mathbb{B}}f_{c}c, \quad f_{c}\in \A;\quad \beta b''\beta'=\sum_{c\in \mathbb{B}}g_{c}c, \quad g_{c}\in \A).
\end{align*}
\end{lem}
\begin{proof}
The sufficiency is obvious. Assume that $b'\preceq_{L} b''$, by definition there exists a sequence $b_1=b'', b_2,\ldots, b_n=b'$ in $\mathbb{B}$ and a sequence $\beta_1,\ldots, \beta_{n-1}$ in $\mathbb{B}$ such that $h_{\beta_i, b_i}^{b_{i+1}}\neq 0$ for $i=1,\ldots, n-1$. Since all $h_{a, b}^{c}\in \mathbb{N}[q, q^{-1}]$, we see that when expanding $\beta_{n-1}\cdots \beta_{1}b''$, the coefficient of $b'$ is $h_{\beta_1, b''}^{b_{2}}\cdots h_{\beta_{n-1}, b_{n-1}}^{b'}+f(q)$ for some $f(q)\in \mathbb{N}[q, q^{-1}]$, which must be nonzero. Therefore, there exists some $\beta\in \mathbb{B}$ such that the coefficient $e_{b'}$ of $b'$ is nonzero when expanding $\beta b''$.
\end{proof}

\section{Description of two-sided cells in $S^{\jmath}(n, d)$}
\label{c5}
\subsection{Preliminaries}
\label{sec:module}

Fix $d\in \Z_{\geq 1}$. Let $W_{B_d}$ be the Weyl group of type $B_d$ with generators $S=\{s_{0},s_{1}, \ldots, s_{d-1}\}$ and the identity element $e$. Let $\HH=\HH_{B_d}$ be the associated {\em Iwahori--Hecke algebra} over $\A$ with a basis $\{T_{w}\:|\:w\in W_{B_d}\}$ satisfying the following relations:
\begin{align}
\label{Btype-Hecke-alg-rela}
\begin{split}
T_{w}T_{w'}&=T_{ww'}\quad \mathrm{if}~  w, w'\in W_{B_d}~ \mathrm{with}~ \ell(ww')=\ell(w)+\ell(w'),\\
T_{s_i}^{2}&=1+(q-q^{-1})T_{s_i}\quad \mathrm{for}~ 0\leq i\leq d-1,
\end{split}
\end{align}
where $\ell$ is the {\em length function} on $W_{B_d}$ (cf. \cite[\S1.1]{Lu03}).


By \cite[Lemma 4.2]{Lu03} there exists a unique $\mathbb{Z}$-algebra automorphism $\bar{\cdot}:\HH\rightarrow \HH$ given by
\[
\overline{q}=q^{-1} \quad\mbox{and}\quad \overline{T_{w}}=T_{w^{-1}}^{-1}~ (\forall~w\in W_{B_d}).\]

Fix $n=2r+1\in \mathbb{Z}_{\geq 3}$. Let $\mathbb{N}=\{0,1,2,\ldots\}$. Following \cite[(2.1.8)]{LL21} we set
$$\Lambda^{\jmath}(n, d) :=\{\lambda=(\lambda_{i})_{-r\leq i\leq r}\in \mathbb{N}^{n}\:|\:\lambda_{-i}=\lambda_{i}\mathrm{~for~}1\leq i\leq r, ~\sum\limits_{i=-r}^{r}\lambda_{i}=2d+1\}.$$
For each $\lambda\in \Lambda^{\jmath}(n, d)$, let $W_{\lambda}$ be the {\em parabolic subgroup} of $W_{B_d}$ generated by the following simple reflections
$$S\setminus\{s_{\lfloor\frac{\lambda_{0}}{2}\rfloor}, s_{\lfloor\frac{\lambda_{0}}{2}\rfloor+\lambda_1},\ldots,s_{\lfloor\frac{\lambda_{0}}{2}\rfloor+\lambda_1+\cdots+\lambda_{r-1}}\},$$ and set $x_{\lambda} :=\sum_{w\in W_{\lambda}}q^{\ell(w)}T_{w}$.

Following \cite[\S2.3]{BKLW18} and also \cite[(3.1.6)]{LL21} we define the {\em $\jmath$-Schur algebra} of type $B_d$ over $\A$ by
$$S^{\jmath}(n, d)=\End_{\HH}(\bigoplus_{\lambda\in \Lambda^{\jmath}(n, d)}x_{\lambda}\HH).$$

For $\lambda, \mu\in \Lambda^{\jmath}(n, d)$, let $\D_{\lambda\mu}$ (resp. $\D_{\lambda\mu}^{+}$) be the set of minimal (resp. maximal) length double coset representatives in $W_\lambda\backslash W_{B_d}/W_\mu$. Set
\begin{equation*}   \label{eq:Xi}
\Xi :=\{(\lambda,g,\mu)~|~\lambda,\mu\in\Lambda^{\jmath}(n, d), g\in\D_{\lambda\mu}\}
\end{equation*}
and
\begin{equation*}   \label{eq:Xi}
\widetilde{\Xi} :=\{(\lambda,g,\mu)~|~\lambda,\mu\in\Lambda^{\jmath}(n, d), g\in\D_{\lambda\mu}^{+}\}.
\end{equation*}
It is well-known that $\Xi$ is in bijection with $\widetilde{\Xi}$ (cf. \cite[Proposition 9.15]{Lu03}) and we shall identify them.

Let $D=2d+1$. Following \cite[(2.2.1-2)]{LL21} we set
$$\Theta_{n,d} :=\{(a_{ij})_{-r\leq i, j\leq r}\in \mathrm{Mat}_{n\times n}(\mathbb{N})\:|\:\sum_{i, j}a_{ij}=D\}$$
and
$$\Pi_{n, d} :=\{(a_{ij})\in \Theta_{n,d}\:|\:a_{ij}=a_{-i,-j}\mathrm{~for~all~}i,j\}.$$

Following \cite[(2.1.9)]{LL21}, for each $\lambda=(\lambda_{i})_{-r\leq i\leq r}\in \Lambda^{\jmath}(n, d)$ and an integer $-r\leq i\leq r$ we define an integer interval $R_{i}^{\lambda}\subset \mathbb{Z}$ by
\begin{align*}
R_{i}^{\lambda}=\left\{
\begin{array}{ll}
\Big[\lfloor\frac{\lambda_{0}}{2}\rfloor+\sum\limits_{1\leq j< i}\lambda_{j}+1, \lfloor\frac{\lambda_{0}}{2}\rfloor+\sum\limits_{1\leq j\leq i}\lambda_{j}\Big]&\hspace{0.35cm}\text{if $0< i\leq r$},\\[0.3em]
\big[-\lfloor\frac{\lambda_{0}}{2}\rfloor, \lfloor\frac{\lambda_{0}}{2}\rfloor\big]&\hspace{0.35cm}\text{if $i=0$},\\[0.3em]
-R_{-i}^{\lambda}&\hspace{0.35cm}\text{if $-r\leq i< 0$}.
\end{array}
\right.
\end{align*}
Note that by \cite[(2.1.10)]{LL21} the parabolic subgroup $W_{\lambda}$ can be characterized by
\begin{align*}
W_{\lambda}=\bigcap_{0\leq i\leq r}\mathrm{Stab}(R_{i}^{\lambda})\quad \text{ for each $\lambda\in \Lambda^{\jmath}(n, d)$},
\end{align*}
where $\mathrm{Stab}(R_{i}^{\lambda})$ denotes the stabilizer of $R_{i}^{\lambda}$ in $W_{B_d}$.

By \cite[Lemma 2.2.1]{LL21} (see also \cite[\S1.3.8-10]{JK81} for the case of the symmetric group), we have the following results.
\begin{lem}\label{a3}
\begin{itemize}
\item[(1)]
For any $\lambda, \mu\in \Lambda^{\jmath}(n, d)$ and $x, y\in W_{B_d}$, we have $x\in W_\lambda y W_\mu$ if and only if $\sharp \big(R_{i}^{\lambda}\cap y(R_{j}^{\mu})\big)=\sharp \big(R_{i}^{\lambda}\cap x(R_{j}^{\mu})\big)$ for all integers $-r\leq i, j\leq r$.
\item[(2)]
Moreover, assuming that $\lambda=(\lambda_{i})_{-r\leq i\leq r}$ and $\mu=(\mu_{i})_{-r\leq i\leq r}$, the map
\begin{align*}
f :W_\lambda y W_\mu\mapsto (a_{ij})_{-r\leq i, j\leq r}, \quad\text{where $a_{ij}=\sharp \big(R_{i}^{\lambda}\cap y(R_{j}^{\mu})\big)$}
\end{align*}
establishes a bijection between the set of double cosets of $W_\lambda$ and $W_\mu$ in $W_{B_d}$ and the set of matrices $(a_{ij})$ in $\Pi_{n, d}$ which satisfy
\begin{align*}
\lambda_{k}=\sum\limits_{-r\leq j\leq r}a_{kj}  \quad\text{and}\quad \mu_{k}=\sum\limits_{-r\leq i\leq r}a_{ik}\quad \text{ for $-r\leq k\leq r$}.
\end{align*}
\end{itemize}
\end{lem}

By Lemma \ref{a3}, there is a natural bijection between $\Xi$ or $\widetilde{\Xi}$ and $\Pi_{n, d}$, and we shall identify them. In the following, if $A\in \Pi_{n, d}$, we shall denote by $(ro(A), w_{A}, co(A))$ and $(ro(A), w_{A}^{+}, co(A))$ the corresponding element in $\Xi$ and $\widetilde{\Xi}$, respectively.

For $\lambda, \mu\in \Lambda^{\jmath}(n, d)$ and $w\in \D_{\lambda\mu}$, we define $\phi_{\lambda, \mu}^{w}\in S^{\jmath}(n, d)$ by
$$\phi_{\lambda, \mu}^{w}(x_{\nu}h)=\delta_{\mu, \nu}\sum_{z\in W_\lambda wW_\mu}q^{\ell(z)}T_{z}h\quad \mathrm{for~}h\in \HH.$$
Then the set $\{\phi_{\lambda, \mu}^{w}\:|\:\lambda, \mu\in \Lambda^{\jmath}(n, d)\mathrm{~and~}w\in \D_{\lambda\mu}\}$ forms an $\A$-basis of $S^{\jmath}(n, d)$ (see \cite[Lemma 3.2]{LL21}).

Following \cite[(2.2.6)]{LL21} we define two subsets $I_{\mathfrak{a}}$ and $I$ in $\mathbb{Z}\times \mathbb{Z}$ by
\begin{align}\label{a1}
I_{\mathfrak{a}}=(\{0\}\times [1,r])\sqcup ([1,r]\times [-r, r])\quad\text{ and }\quad I=I_{\mathfrak{a}}\sqcup \{(0,0)\}.
\end{align}
For each $(i,j)\in I$, we set
\begin{align}\label{a2}
a_{ij}^{\natural}=\left\{
\begin{array}{ll}
\frac{1}{2}(a_{ij}-1)&\hspace{0.35cm}\text{if $(i,j)=(0,0)$},\\[0.3em]
a_{ij}&\hspace{0.35cm}\text{otherwise}.
\end{array}
\right.
\end{align}

Following \cite[(4.2.1)]{LL21} or \cite[(3.16)]{BKLW18}, for each $A\in \Pi_{n, d}$, we define $d_{A}$ by
$$d_{A}=\frac{1}{2}\sum_{(i,j)\in I}\Bigg(\sum\limits_{\substack{x\leq i\\y> j}} a_{ij}^{\natural}a_{xy}+\sum\limits_{\substack{x\geq i\\y< j}} a_{ij}^{\natural}a_{xy}\Bigg).$$
For each $A\in \Pi_{n, d}$, if $A=(\lambda,w_{A},\mu)$, following \cite[(4.2.3)]{LL21} we set
\begin{align}\label{g1}
[A]=q^{-d_{A}}\phi_{\lambda, \mu}^{w_{A}}.
\end{align}
Then the set $\{[A]\:|\:A\in \Pi_{n, d}\}$ also forms an $\A$-basis of $S^{\jmath}(n, d)$, which is called the {\em standard basis}.

We define a bar involution $\bar{\cdot}$ on $S^{\jmath}(n, d)$ as follows: for each $f\in \Hom_{\HH}(x_{\nu}\HH,x_{\lambda}\HH)\subset S^{\jmath}(n, d)$, we have
\begin{align}\label{d2}
\bar{f}(x_{\nu'}h)=\delta_{\nu', \nu}\overline{f(\overline{x_{\nu}})}h\quad \mathrm{for~}h\in \HH.
\end{align}
On $\Pi_{n, d}$ we define a partial order denoted by $<$ as follows:
\begin{align}\label{d3}
A'<A \mbox{ if and only if } ro(A')=ro(A), co(A')=co(A), w_{A'}^{+}<w_{A}^{+},
\end{align}
where the $<$ on $W_{B_d}$ is the usual Bruhat ordering (cf. \cite[\S2.1]{Lu03}). Recall that in \cite[\S3.6]{BKLW18}, a {\em canonical basis} for $S^{\jmath}(n, d)$ has been constructed, which is denoted by $\{\{A\}\:|\:A\in \Pi_{n, d}\}$. For each $A\in \Pi_{n, d}$, $\{A\}$ is characterized by
$$\overline{\{A\}}=\{A\}$$
and
$$\{A\}=\sum\limits_{A'\leq A}P_{A', A}[A'],$$
where $P_{A, A}=1$ and $P_{A', A}\in q^{-1}\mathbb{N}[q^{-1}]$ for $A'< A$.

Recall that in \cite{KL79}, the left, right and two-sided cells with respect to the Kazhdan--Lusztig basis $\{\mathcal{C}_{w}\:|\:w\in W_{B_d}\}$ of $\HH$ have been defined. In \cite[(2.14)]{Sp82} (see also \cite[(3.2.1)]{Lu85}), it has been proved that the structure constants $h_{x, y}^{z}$ of $\HH$ with respect to the Kazhdan--Lusztig basis $\{\mathcal{C}_{w}\}$ lie in $\mathbb{N}[q, q^{-1}]$. Because of the intersection cohomology construction of $\{\{A\}\:|\:A\in \Pi_{n, d}\}$, the structure constants $g_{A, B}^{C}$ of $S^{\jmath}(n, d)$ associated to them also lie in $\mathbb{N}[q, q^{-1}]$ (see the proof of \cite[Theorem 5.6]{LiW18}). Using these and Lemma \ref{another-charact-cells}, we can obtain a classification of left, right and two-sided cells for $S^{\jmath}(n, d)$ with respect to the canonical basis $\{\{A\}\:|\:A\in \Pi_{n, d}\}$. (In the following, we shall write $\mathcal{C}_{y}\preceq_{\star}\mathcal{C}_{w}$ as $y\preceq_{\star}w$ and $\mathcal{C}_{y}\sim_{\star}\mathcal{C}_{w}$ as $y\sim_{\star}w$ for $\star\in \{L, R, LR\}$.)
\begin{prop}
\label{lem:d=d2=d000ad}
For $A, B\in \Pi_{n, d}$, we have the following results.
\begin{enumerate}
\item
$\{A\}\preceq_{L}\{B\}$ if and only if $co(A)=co(B)$ and $w_{A}^{+}\preceq_{L} w_{B}^{+}$. Similarly, $\{A\}\preceq_{R}\{B\}$ if and only if $ro(A)=ro(B)$ and $w_{A}^{+}\preceq_{R} w_{B}^{+}$.
\item
$\{A\}\sim_{L}\{B\}$ if and only if $co(A)=co(B)$ and $w_{A}^{+}\sim_{L} w_{B}^{+}$. Similarly, $\{A\}\sim_{R}\{B\}$ if and only if $ro(A)=ro(B)$ and $w_{A}^{+}\sim_{R} w_{B}^{+}$.
\item
$\{A\}\sim_{LR}\{B\}$ if and only if $w_{A}^{+}\sim_{LR} w_{B}^{+}$.
\end{enumerate}
\end{prop}
\begin{proof}
The proof is similar to that of \cite[Lemma 2.2 and Corollary 2.3]{Du96}, and we omit the details.
\end{proof}
\subsection{Description of two-sided cells in $\HH$}
\label{d5}
In this subsection, we recall the description of two-sided cells in $\HH$ with respect to the Kazhdan--Lusztig basis following \cite{BV82}.

According to \cite[\S3]{Lu77} and also \cite[\S5]{Lu79}, we define a {\em symbol} in type $B_{d}$ to be an array of nonnegative integers
\begin{equation*}  \label{dim:Sirrep1}
\Lambda={\lambda_1<\lambda_2<\cdots <\lambda_{m+1} \choose \mu_1<\mu_2<\cdots <\mu_{m}}
\end{equation*}
such that $\sum_{i=1}^{m+1}\lambda_{i}+\sum_{j=1}^{m}\mu_{j}=d+m^{2}$, where $m\in \mathbb{Z}_{\geq 0}$. We introduce an equivalence relation $\sim$ on the set of symbols as the transitive closure of the `shift relations'
\begin{equation*}  \label{dim:Sirrep-2}
{\lambda_1<\lambda_2<\cdots <\lambda_{m+1} \choose \mu_1<\mu_2<\cdots <\mu_{m}}=\Lambda\sim\Lambda'={0<\lambda_1+1<\lambda_2+1<\cdots <\lambda_{m+1}+1 \choose 0<\mu_1+1<\mu_2+1<\cdots <\mu_{m}+1}.
\end{equation*}
We shall denote by $[\Lambda]$ the equivalence class of a symbol $\Lambda$ and $\Phi_{d}$ the set of equivalence classes of symbols relative to $\sim$. Let $\Psi_{d}$ be the set of ordered pairs $(\alpha, \beta)$ of partitions $\alpha=(\alpha_{m+1}, \ldots, \alpha_2, \alpha_1)$, $\beta=(\beta_m, \ldots, \beta_2, \beta_1)$ such that $\sum_{i=1}^{m+1}\alpha_{i}+\sum_{j=1}^{m}\beta_{j}=d, \alpha_1\geq 0, \beta_{1}\geq 0$. Then $\Psi_{d}$ is in one-to-one correspondence with $\Phi_{d}$ by associating to $(\alpha, \beta)$ the equivalence class $[\Lambda]$ of a symbol $\Lambda$ defined by $\lambda_i=\alpha_i+i-1$ $(1\leq i\leq m+1)$, $\mu_j=\beta_j+j-1$ $(1\leq j\leq m)$. Both sets are in one-to-one correspondence with the set of complex irreducible representations of $W_{B_d}$ $($up to isomorphism$)$ (cf. \cite[\S5]{Lu79}).

Given a symbol
\begin{align*}
\Lambda={\lambda_1<\lambda_2<\cdots <\lambda_{m+1} \choose \mu_1<\mu_2<\cdots <\mu_{m}}
\end{align*}
for some $m\in \mathbb{Z}_{\geq 0}$, we take the set $\{2\lambda_{i}+1, 2\mu_{j}\:|\:1\leq i\leq m+1, 1\leq j\leq m\}$ and order it in a decreasing sequence, say $(\nu_{2m+1},\ldots,\nu_{1})$. Since $\sum_{i=1}^{2m+1}\nu_{i}=2d+2m^{2}+m+1$ and $\nu_{2m+1}>\nu_{2m}>\cdots>\nu_{1}$, we see that $(\nu_{2m+1}-(2m+1)+1, \nu_{2m}-2m+1,\ldots,\nu_1-1+1)$ is a partition of $2d+1$, which we shall denote by $par(\Lambda)$. By definition, it is easy to see that if $\Lambda\sim\Lambda'$, then $par(\Lambda)$ differs from $par(\Lambda')$ by possibly some 0's, which can be regarded as the same partition and we denote it by $par[\Lambda]$. Let $\mathfrak{P}$ be the set of partitions $par[\Lambda]$, where $[\Lambda]$ runs over $\Phi_{d}$.

We now recall the Robinson--Schensted algorithm following \cite[p.171]{BV82}. Let $\lambda=(\lambda_1, \lambda_2,\ldots,\lambda_{l(\lambda)})$ be a partition of $2d+1$, where $l(\lambda)$ is the number of nonzero components of $\lambda$. The Young diagram of $\lambda$ is a collection of boxes, arranged in left justified rows with $\lambda_1$ boxes in row 1, $\lambda_2$ boxes in row 2, and so on. A $\lambda$-tableau $T$ is obtained by filling in the boxes in the Young diagram of $\lambda$ by the integers $1,2,\ldots,2d+1$, each occurring once; we will call $\lambda$ the {\em shape} of $T$. A $\lambda$-tableau $T$ is called {\em standard} if the entries decrease along each row and down each column.

For a standard tableau $T$ and a positive integer $k$, we construct a new tableau $T\leftarrow k$ by the following recursive algorithm:
\begin{enumerate}
\item
If $k_1=k$ is less than or equal to any integer in the first row of $T$, then $T\leftarrow k$ is the tableau obtained by adding $k_1$ at the end of the first row of $T$.
\item
Otherwise, find the largest entry $k_2$ in the first row satisfying $k_2< k_1$, and replace $k_2$ by $k_1$.
\item
Repeat the procedure for the second row of $T$ with $k_2$, and so on.
\end{enumerate}
\noindent The algorithm ends whenever an integer $k_a$ is added at the end of the $a$-th row.

Given an element of $\mathfrak{S}_{2d+1}$:
\begin{equation*}
w=\left(\hspace{-1.6mm}
 \begin{array}{ccccccccccccccccccccccccc}
 1&2&\cdots&2d+1\\
 w_1&w_2&\cdots&w_{2d+1}\\
\end{array}
\hspace{-1.6mm}\right),
\end{equation*}
we define a standard tableau $T(w)$ recursively as follows:
$$P_0=\emptyset, \quad P_{t}=P_{t-1}\leftarrow w_{t} \mbox{ for } 1\leq t\leq 2d+1,\quad T(w) :=P_{2d+1}.$$
Then the map $w\mapsto (T(w), T(w^{-1}))$ gives a bijection between $\mathfrak{S}_{2d+1}$ and the set of all pairs of standard tableaux of the same shape, which is called the {\em Robinson--Schensted correspondence}.

We shall identify $W_{B_d}$ with the group of permutations $w$ on the set $\{-d, -(d-1),\ldots,-1,0,1,$ $\ldots,d-1,d\}$ such that $w(-i)=-w(i)$ for any $i$. Under the identification, we have
\begin{align*}
s_{0}=(-1, 1),\quad s_{i}=(-i-1, -i)(i, i+1) \text{ for $1\leq i\leq d-1$.}
\end{align*}
In particular, $w(0)=0$ for any $w\in W_{B_d}$. For each $w\in W_{B_d}$, it can be considered as an element of $\mathfrak{S}_{2d+1}$ under the identification $-d\leftrightarrow 1, -(d-1)\leftrightarrow 2,\ldots, 0\leftrightarrow d+1,\ldots, d-1\leftrightarrow 2d, d\leftrightarrow 2d+1$, and therefore we can associate to it a pair of standard tableaux $(T(w), T(w^{-1}))$ by applying the Robinson--Schensted correspondence for $\mathfrak{S}_{2d+1}$ as above (cf. \cite[p.173]{BV82}); let us denote by $PT(w)$ the shape of $T(w)$, which is a partition of $2d+1$. Then we have the following proposition due to Barbasch and Vogan.
\begin{prop} $($see \cite[Proposition 17]{BV82}$)$
\label{lem:d=d2=d000adadc}
For any $w\in W_{B_d}$, $PT(w)$ belongs to the set $\mathfrak{P}.$
\end{prop}

We now define an equivalence relation $\approx$ on $\Phi_{d}$. Given two elements of $\Phi_{d}$:
$$[\Lambda]=\bigg[{\lambda_1<\lambda_2<\cdots <\lambda_{m+1} \choose \mu_1<\mu_2<\cdots <\mu_{m}}\bigg]\quad\mbox{and} \quad [\Lambda']=\bigg[{\lambda_1'<\lambda_2'<\cdots <\lambda_{m+1}' \choose \mu_1'<\mu_2'<\cdots <\mu_{m}'}\bigg],$$ we say that $[\Lambda]\approx[\Lambda']$ if and only if $$\{\lambda_1, \lambda_2, \ldots, \lambda_{m+1}, \mu_1, \mu_2, \ldots, \mu_{m}\}\mbox{ is a permutation of }\{\lambda_1', \lambda_2', \ldots, \lambda_{m+1}', \mu_1', \mu_2', \ldots, \mu_{m}'\}.$$
It is easy to see that if $par[\Lambda]=par[\Lambda']$, then $[\Lambda]=[\Lambda']$ (see \cite[p.80]{Mc96}). Thus, the equivalence relation $\approx$ on $\Phi_{d}$ induces an equivalence relation, which we shall denote by the same notation $\approx$, on the set $\mathfrak{P}$; we say that $par[\Lambda]\approx par[\Lambda']$ if and only if $[\Lambda]\approx[\Lambda']$. Then we have the following proposition, which gives an explicit description of two-sided cells in $\HH$.
\begin{prop} $($see \cite[Theorem 18]{BV82}$)$
\label{theor:d=d2=d000adadc}
For any two elements $w, w'\in W_{B_d}$, we have $w\sim_{LR} w'$ if and only if $PT(w)\approx PT(w')$.
\end{prop}

Let $\mathcal{G}_{W_{B}}$ be the subset of $\Phi_{d}$ consisting of equivalence classes of symbols such that $\lambda_{i}\leq \mu_{i}\leq\lambda_{i+1}$ for any $i$. We shall call each element of $\mathcal{G}_{W_{B}}$ {\em special}, which exactly corresponds to a so-called special representation of $W_{B_d}$ (see \cite{Lu79} and also \cite{Lu82}). Let $\tilde{\mathcal{G}}_{W_{B}}$ be the set of equivalence classes of $\Phi_{d}$ relative to $\approx$. It is obvious that each equivalence class in $\tilde{\mathcal{G}}_{W_{B}}$ contains exactly one element of $\mathcal{G}_{W_{B}}$.

Let $\mathcal{P}_{d}$ be the set of partitions of $2d+1$ such that every {\em even} part occurs an even number of times. It is well-known that there exists a one-to-one correspondence between $\mathcal{P}_{d}$ and the set of nilpotent orbits in type $B_{d}$ by the Jordan block decomposition (cf. \cite[Theorem 5.1.2]{CoMc93}). We call an element $\lambda\in \mathcal{P}_{d}$ a {\em special partition} if its conjugate $\lambda^{t}$ also belongs to $\mathcal{P}_{d}$ (cf. \cite[\S6.3]{CoMc93}). We denote by $\mathcal{SP}_{d}$ the set of special partitions of $2d+1$. Then it is easy to check that the map, $\pi:[\Lambda]\mapsto par[\Lambda]$, gives a one-to-one correspondence between $\mathcal{G}_{W_{B}}$ and $\mathcal{SP}_{d}$. (We refer to \cite[p.80]{Mc96} for a construction of the inverse of $\pi$, which is denoted by $\pi'_{n}$.) Thus, by Proposition \ref{theor:d=d2=d000adadc}, we see that there is a one-to-one correspondence between the set of two-sided cells in $\HH$ and $\mathcal{SP}_{d}$.


\subsection{Description of two-sided cells in $S^{\jmath}(n, d)$}
\label{d4}
In this subsection, we shall give a combinatorial description of two-sided cells in $S^{\jmath}(n, d)$.

We first recall a result of Du, which associates partitions to the matrices in $\Pi_{n, d}$ and generalizes Greene's method \cite{Gr79} associating partitions to finite partially ordered sets. Let $P$ be a finite partially ordered set. A {\em chain} in $P$ is a subset of $P$ which is totally ordered by the induced order of $P$. A {\em $k$-chain family} is a subset of P which is a disjoint union of $k$ chains. Let $[-r, r]\subset \mathbb{Z}$ denote the interval from $-r$ to $r$ in $\mathbb{Z}$ and set $[-r, r]^{2}=[-r, r]\times [-r, r]$. It is easy to see that $[-r, r]^{2}$ is a finite partially ordered set with the induced order by setting $(i, j)\leq (i', j')$ if $i\geq i'$ and $j\leq j'$.

Let $\mathfrak{s} :\Pi_{n, d}\rightarrow \mathbb{N}$ be the map sending a matrix $A=(a_{ij})_{-r\leq i,j\leq r}\in \Pi_{n, d}$ to its entry sum $\mathfrak{s}(A)=\sum_{i,j}a_{ij}$. Moreover generally, if $F$ is a $k$-chain family of $[-r, r]^{2}$, we define $\mathfrak{s}_{F}(A)$ to be the sum of the entries $a_{ij}$ with $(i, j)\in F$. We call the map $\mathfrak{s}_{F} :\Pi_{n, d}\rightarrow \mathbb{N}$ an $F$-sum map and $\mathfrak{s}_{F}(A)$ the $F$-sum of $A$. Let $\mathfrak{s}_{k}(A)$ be the {\em maximum} value of $F$-sums of $A$ for all $k$-chain families $F$. We have the following result due to Du.

\begin{thm} $($see \cite[Theorem 1.2]{Du96}$)$
\label{theor:d=d2=d000=du-cells}
For each $A\in \Pi_{n, d}$, we define $\sigma_{i}(A)=\mathfrak{s}_{i}(A)-\mathfrak{s}_{i-1}(A)$ (with the convention that $\mathfrak{s}_{0}(A)=0$) for all $1\leq i\leq n$. Then $\sigma(A)=(\sigma_{1}(A), \sigma_{2}(A), \ldots, \sigma_{n}(A))$ is a partition of $2d+1$.
\end{thm}
For each $A\in \Pi_{n, d}$, we can associate a permutation $y_{A}$ to it as follows. We first construct a pseudo-matrix $A_{+}=(c_{ij})$, where the $c_{ij}$'s are all sequences of numbers in $\{-d, -(d-1),\ldots,-1,0,1,\ldots,d-1,d\}$ such that the numbers $\{-d, -(d-1), \ldots, -1, 0, 1, \ldots, d-1, d\}$ in $A_{+}$ are arranged in the way that they are in the natural order when read from right to left inside the sequence and from right to left along the rows, followed by top to bottom down the successive rows, and there are $a_{ij}$ elements in each sequence $c_{ij}$. Then the permutation $y_{A}$ may be read off from $A_{+}$ by reading from left to right inside the sequences and from bottom to top up the columns, followed by left to right along the successive columns.

The following lemma can be proved in a similar way as \cite[Lemma 3.2]{Du96} (see also \cite[Lemma 2.2.2]{LL21}).
\begin{lem}\label{longest element}
For each $A=(a_{ij})=(ro(A), w_{A}^{+}, co(A))\in \Pi_{n, d}$, we have $y_{A}=w_{A}^{+}$, and moreover,
\begin{align}\label{b3}
\ell(w_{A}^{+})=d^{2}-\frac{1}{2}\sum_{(i,j)\in I}\Bigg(\sum\limits_{\substack{x< i\\y< j}}a_{ij}^{\natural}a_{xy}+\sum\limits_{\substack{x> i\\y> j}}a_{ij}^{\natural}a_{xy}\Bigg).
\end{align}
\end{lem}
\begin{proof}
Recall the two sets $I_{\mathfrak{a}}$ and $I$ in \eqref{a1}. We define a subset $I_{\mathfrak{a}}^{-}$ in $\mathbb{Z}\times \mathbb{Z}$ by
\begin{align*}
I_{\mathfrak{a}}^{-}=(\{0\}\times [-r, -1])\sqcup ([-r, -1]\times [-r, r]).
\end{align*}
We set
\begin{align*}
N=&\sum_{(i,j)\in I}\Bigg(\sum\limits_{\substack{x< i\\y< j}}a_{ij}^{\natural}a_{xy}+\sum\limits_{\substack{x> i\\y> j}}a_{ij}^{\natural}a_{xy}\Bigg)\\
=&\sum_{(i,j)\in I_{\mathfrak{a}}}\sum\limits_{\substack{x< i\\y< j}}a_{ij}a_{xy}+\sum_{(i,j)\in I_{\mathfrak{a}}}\sum\limits_{\substack{x> i\\y> j}}a_{ij}a_{xy}\\
&+\frac{1}{2}\sum\limits_{\substack{x< 0\\y< 0}}a_{00}a_{xy}+\frac{1}{2}\sum\limits_{\substack{x> 0\\y> 0}}a_{00}a_{xy}-\frac{1}{2}\sum\limits_{\substack{x< 0\\y< 0}}a_{xy}-\frac{1}{2}\sum\limits_{\substack{x> 0\\y> 0}}a_{xy}. \quad \text{ (by \eqref{a2})}
\end{align*}

Let $w$ be an element in the double coset $W_\lambda w_{A}^{+} W_\mu$ corresponding to $A$, where $\lambda=ro(A), \mu=co(A)$. By Lemma \ref{a3}, we have
\begin{align*}
\sharp \big(R_{i}^{\lambda}\cap w(R_{j}^{\mu})\big)= \sharp \big(R_{i}^{\lambda}\cap w_{A}^{+}(R_{j}^{\mu})\big)=a_{ij}~~\text{and}~~
\sharp \big(R_{x}^{\lambda}\cap w(R_{y}^{\mu})\big)= \sharp \big(R_{x}^{\lambda}\cap w_{A}^{+}(R_{y}^{\mu})\big)=a_{xy}.
\end{align*}
For any $-r\leq i< x\leq r$ and $-r\leq j< y\leq r$, we pick $c\in R_{i}^{\lambda}\cap w(R_{j}^{\mu})$ and $d\in R_{x}^{\lambda}\cap w(R_{y}^{\mu})$. Obviously, there are
\begin{align*}
N_1=\sum\limits_{\substack{-r\leq i< x\leq r\\-r\leq j< y\leq r}}a_{ij}a_{xy}
\end{align*}
ways to make such selection. Let $a=w^{-1}(c), b=w^{-1}(d)$. Then $-r\leq i< x\leq r$ and $-r\leq j< y\leq r$ imply that $-d\leq w(a)< w(b)\leq d$ and $-d\leq a< b\leq d$, respectively. So there are at least $N_1$ pairs $(a, b)$ satisfying $-d\leq a< b\leq d$ and $-d\leq w(a)< w(b)\leq d$, that is,
\begin{align}\label{a4}
N_1\leq \sharp \big\{(a, b)\in [-d, d]\times[-d, d]\:\big|\:a< b\text{ and }w(a)< w(b)\big\},
\end{align}
where $[-d, d]\subset \mathbb{Z}$ denotes the interval from $-d$ to $d$ in $\mathbb{Z}$. Similarly, we have
\begin{align}\label{a5}
N_2=\sum\limits_{\substack{-r\leq x< i\leq r\\-r\leq y< j\leq r}}a_{ij}a_{xy} \leq \sharp \big\{(a, b)\in [-d, d]\times[-d, d]\:\big|\:a> b\text{ and }w(a)> w(b)\big\}.
\end{align}

It is obvious that
\begin{align*}
N_1+N_2=&\sum_{(i,j)\in I_{\mathfrak{a}}}\sum\limits_{\substack{x> i\\y> j}}a_{ij}a_{xy}+ \sum_{(i,j)\in I_{\mathfrak{a}}^{-}}\sum\limits_{\substack{x> i\\y> j}}a_{ij}a_{xy}+\sum\limits_{\substack{x> 0\\y> 0}}a_{00}a_{xy}\\
&+\sum_{(i,j)\in I_{\mathfrak{a}}}\sum\limits_{\substack{x< i\\y< j}}a_{ij}a_{xy}+ \sum_{(i,j)\in I_{\mathfrak{a}}^{-}}\sum\limits_{\substack{x< i\\y< j}}a_{ij}a_{xy}+\sum\limits_{\substack{x< 0\\y< 0}}a_{00}a_{xy}.
\end{align*}
Since $A\in \Pi_{n, d}$, we have $a_{ij}=a_{-i,-j}$ for all $i,j$, which implies that
\begin{align*}
\sum_{(i,j)\in I_{\mathfrak{a}}}\sum\limits_{\substack{x> i\\y> j}}a_{ij}a_{xy}=\sum_{(i,j)\in I_{\mathfrak{a}}^{-}}\sum\limits_{\substack{x< i\\y< j}}a_{ij}a_{xy}, \quad \sum_{(i,j)\in I_{\mathfrak{a}}^{-}}\sum\limits_{\substack{x> i\\y> j}}a_{ij}a_{xy}=\sum_{(i,j)\in I_{\mathfrak{a}}}\sum\limits_{\substack{x< i\\y< j}}a_{ij}a_{xy}
\end{align*}
and
\begin{align*}
R=\sum\limits_{\substack{x> 0\\y> 0}}a_{xy}=\sum\limits_{\substack{x< 0\\y< 0}}a_{xy}.
\end{align*}
Therefore, we have $N_1+N_2=2N+2R$.

On the other hand, by $w(-i)=-w(i)$ for any $i\in [-d, d]$ we have
\begin{align*}
(2d+1)^{2}=&\sharp \big\{(a, b)\in [-d, d]\times[-d, d]\big\}\\
=&\sharp \Big\{ (a, b)\in [1, d]\times[-d, d] ~\Big|~ \substack{a< b\\w(a)> w(b)} \text{ or } \substack{a> b\\w(a)< w(b)} \Big\}\\
&+\sharp \Big\{ (a, b)\in [1, d]\times[-d, d] \:\Big|\: \substack{a< b\\w(a)< w(b)} \text{ or } \substack{a> b\\w(a)> w(b)} \Big\}\\
&+\sharp \big\{(a, b)\in [1, d]\times[-d, d]\hspace{0.3mm}\big|\hspace{0.3mm}a=b\big\}\hspace{-0.5mm}+\hspace{-0.5mm}\sharp \big\{(a, b)\in [-d, -1]\times[-d, d]\hspace{0.3mm}\big|\hspace{0.3mm}a=b\big\}\\
&+\sharp \Big\{ (a, b)\in [-d, -1]\times[-d, d] \:\Big|\: \substack{a< b\\w(a)> w(b)} \text{ or } \substack{a> b\\w(a)< w(b)} \Big\}\\
&+\sharp \Big\{ (a, b)\in [-d, -1]\times[-d, d] \:\Big|\: \substack{a< b\\w(a)< w(b)} \text{ or } \substack{a> b\\w(a)> w(b)} \Big\}\\
&+\sharp \Big\{ (a, b)\in \{0\}\times[-d, -1] \hspace{0.5mm}\Big|\hspace{0.5mm} \substack{0> b\\0> w(b)} \Big\}\hspace{-0.3mm}+\hspace{-0.3mm}\sharp \Big\{ (a, b)\in \{0\}\times[-d, -1] \hspace{0.5mm}\Big|\hspace{0.5mm} \substack{0> b\\0< w(b)} \Big\}\\
&+\sharp \Big\{ (a, b)\in \{0\}\times[1, d] \:\Big|\: \substack{0< b\\0> w(b)} \Big\}+ \sharp \Big\{ (a, b)\in \{0\}\times[1, d] \:\Big|\: \substack{0< b\\0< w(b)} \Big\}\\
&+\sharp \Big\{ (a, b)\in \{0\}\times\{0\} \Big\}\\
=&2\sharp \Big\{ (a, b)\in [1, d]\times[-d, d] \:\Big|\: \substack{a< b\\w(a)> w(b)} \text{ or } \substack{a> b\\w(a)< w(b)} \Big\}\\
&+2\sharp \Big\{ (a, b)\in [1, d]\times[-d, d] \:\Big|\: \substack{a< b\\w(a)< w(b)} \text{ or } \substack{a> b\\w(a)> w(b)} \Big\}\\
&+2\sharp \big\{b\in [1, d]\:\big|\:0> w(b)\big\}+2\sharp \big\{b\in [1, d]\:\big|\:0< w(b) \big\}+2d+1.
\end{align*}
By \cite[(2.1.7)]{LL21}, we have
\begin{align*}
2\ell(w)=\sharp \Big\{ (a, b)\in [1, d]\times[-d, d] \:\Big|\: \substack{a< b\\w(a)> w(b)} \text{ or } \substack{a> b\\w(a)< w(b)} \Big\}.
\end{align*}
We set
\begin{align*}
P_{w}\hspace{-0.15mm}=\hspace{-0.15mm}\sharp \Big\{ (a, b)\in [1, d]\hspace{-0.3mm}\times\hspace{-0.3mm}[-d, d] \hspace{0.3mm}\Big|\hspace{0.3mm} \substack{a< b\\w(a)< w(b)} \text{ or } \substack{a> b\\w(a)> w(b)} \Big\}~~\text{and}~~ Q_{w}\hspace{-0.15mm}=\hspace{-0.15mm}\sharp \big\{b\in [1, d]\hspace{0.3mm}\big|\hspace{0.3mm}0< w(b) \big\}.
\end{align*}
Therefore, we have
\begin{align}\label{a6}
4d^{2}+2d=4\ell(w)+2P_{w}+2Q_{w}+2\sharp \big\{b\in [1, d]\:\big|\:0> w(b)\big\}.
\end{align}

By \eqref{a4} and \eqref{a5}, we have
\begin{align*}
N_1+N_2\leq & \sharp \big\{(a, b)\in [-d, d]\times[-d, d]\:\big|\:a< b\text{ and }w(a)< w(b)\big\} \\
&+\sharp \big\{(a, b)\in [-d, d]\times[-d, d]\:\big|\:a> b\text{ and }w(a)> w(b)\big\}=2P_{w}+2Q_{w}.
\end{align*}
Combining $N_1+N_2=2N+2R$ with \eqref{a6}, we obtain
\begin{align}\label{b1}
4\ell(w)+2\sharp \big\{b\in [1, d]\:\big|\:0> w(b)\big\}+2N+2R\leq 4d^{2}+2d.
\end{align}

Next we shall prove that
\begin{align}\label{b2}
y_{A}\in W_\lambda w_{A}^{+} W_\mu\quad\text{ and }\quad 2P_{y_{A}}+2Q_{y_{A}}=N_1+N_2.
\end{align}
Let $C_{ij}=C_{ij}(A)$ denote the set of elements in the sequence $c_{ij}$. From the definition of $y_{A}$, we see that
\begin{align*}
\bigcup_{-r\leq l\leq r}C_{il}=R_{i}^{\lambda}\quad\text{and}\quad \bigcup_{-r\leq k\leq r}C_{kj}=y_{A}(R_{j}^{\mu}).
\end{align*}
Thus, we have $C_{ij}=R_{i}^{\lambda}\cap y_{A}(R_{j}^{\mu})$, and hence $a_{ij}=\sharp \big(R_{i}^{\lambda}\cap y_{A}(R_{j}^{\mu})\big)$, which implies that $y_{A}\in W_\lambda w_{A}^{+} W_\mu$ by Lemma \ref{a3}.

Note that $y_{A}^{-1}(C_{ij}(A))=C_{ji}(A^{\mathrm{tr}})$, where $A^{\mathrm{tr}}$ denotes the transpose of $A$. Consider $a, b$ such that $-d\leq a< b\leq d$ and $y_{A}(a)< y_{A}(b)$. There exist $i, j$ such that $y_{A}(a)\in C_{ij}$ and $x, y$ such that $y_{A}(b)\in C_{xy}$. Note that $a\in C_{ji}(A^{\mathrm{tr}})$ and $b\in C_{yx}(A^{\mathrm{tr}})$. Since $a< b$, we have either $j< y$ or $j=y$ and $i> x$, while since $y_{A}(a)< y_{A}(b)$, we have either $i< x$ or $i=x$ and $j> y$. These together imply that $i< x$ and $j< y$. Thus, $\sharp \big\{(a, b)\in [-d, d]\times[-d, d]\:\big|\:a< b\text{ and }y_{A}(a)< y_{A}(b)\big\}\leq N_1$. By \eqref{a4} we obtain that $\sharp \big\{(a, b)\in [-d, d]\times[-d, d]\:\big|\:a< b\text{ and }y_{A}(a)< y_{A}(b)\big\}=N_1$. Similarly, we have $\sharp \big\{(a, b)\in [-d, d]\times[-d, d]\:\big|\:a> b\text{ and }y_{A}(a)> y_{A}(b)\big\}=N_2$. Therefore, we have $2P_{y_{A}}+2Q_{y_{A}}=N_1+N_2$.

Combining $N_1+N_2=2N+2R$ with \eqref{a6}, \eqref{b1} and \eqref{b2}, we obtain
\begin{align*}
4\ell(w_{A}^{+})+2\sharp \big\{b\in [1, d]&\:\big|\:0> w_{A}^{+}(b)\big\}+2N+2R\\
&\leq 4\ell(y_{A})+2\sharp \big\{b\in [1, d]\:\big|\:0> y_{A}(b)\big\}+2N+2R,
\end{align*}
which implies that $y_{A}=w_{A}^{+}$ by the uniqueness of the longest element in the double coset $W_\lambda w_{A}^{+} W_\mu$ (cf. \cite[Proposition 9.15(e)]{Lu03}). From the definition of $y_{A}$, we see that $R=Q_{y_{A}}$. Combining this with $N_1+N_2=2N+2R$, \eqref{a6} and \eqref{b2}, we obtain
\begin{align*}
4d^{2}+2d=&4\ell(y_{A})+2P_{y_{A}}+2Q_{y_{A}}+2\sharp \big\{b\in [1, d]\:\big|\:0> y_{A}(b)\big\}\\
=&4\ell(y_{A})+2N+2Q_{y_{A}}+2\sharp \big\{b\in [1, d]\:\big|\:0> y_{A}(b)\big\}\\
=&4\ell(y_{A})+2N+2d,
\end{align*}
that is, $\ell(w_{A}^{+})=\ell(y_{A})=d^{2}-\frac{1}{2}N$. We are done.
\end{proof}

\begin{example}\label{b4}
Assume that $n=3, r=1$ and $d=2$. By \cite[Lemma 2.2]{BKLW18}, we have $\sharp\Pi_{3, 2}={2+2+2 \choose 2}=15$. We completely describe the $15$ elements in $\Pi_{3, 2}$ as follows:
\begin{align*}
A_1=\left(\hspace{-1mm}
 \begin{array}{ccc}
 2 & 0 & 0\\
 0 & 1 & 0\\
 0&  0 & 2\\
 \end{array}
\hspace{-1mm}\right),~~
A_2=\left(\hspace{-1mm}
 \begin{array}{ccc}
 0 & 2 & 0\\
 0 & 1 & 0\\
 0&  2 & 0\\
 \end{array}
\hspace{-1mm}\right),~~
A_3=\left(\hspace{-1mm}
 \begin{array}{ccc}
 0 & 0 & 2\\
 0 & 1 & 0\\
 2&  0 & 0\\
 \end{array}
\hspace{-1mm}\right),~~
A_4=\left(\hspace{-1mm}
 \begin{array}{ccc}
 0 & 0 & 0\\
 2 & 1 & 2\\
 0&  0 & 0\\
 \end{array}
\hspace{-1mm}\right),
\end{align*}
\begin{align*}
A_5=\left(\hspace{-1mm}
 \begin{array}{ccc}
 1 & 1 & 0\\
 0 & 1 & 0\\
 0&  1 & 1\\
 \end{array}
\hspace{-1mm}\right),~~
A_6=\left(\hspace{-1mm}
 \begin{array}{ccc}
 1 & 0 & 1\\
 0 & 1 & 0\\
 1&  0 & 1\\
 \end{array}
\hspace{-1mm}\right),~~
A_7=\left(\hspace{-1mm}
 \begin{array}{ccc}
 1 & 0 & 0\\
 1 & 1 & 1\\
 0&  0 & 1\\
 \end{array}
\hspace{-1mm}\right),~~
A_8=\left(\hspace{-1mm}
 \begin{array}{ccc}
 0 & 1 & 1\\
 0 & 1 & 0\\
 1&  1 & 0\\
 \end{array}
\hspace{-1mm}\right),
\end{align*}
\begin{align*}
A_9=\left(\hspace{-1mm}
 \begin{array}{ccc}
 0 & 1 & 0\\
 1 & 1 & 1\\
 0&  1 & 0\\
 \end{array}
\hspace{-1mm}\right),~~
A_{10}=\left(\hspace{-1mm}
 \begin{array}{ccc}
 0 & 0 & 1\\
 1 & 1 & 1\\
 1&  0 & 0\\
 \end{array}
\hspace{-1mm}\right),~~
A_{11}=\left(\hspace{-1mm}
 \begin{array}{ccc}
 1 & 0 & 0\\
 0 & 3 & 0\\
 0&  0 & 1\\
 \end{array}
\hspace{-1mm}\right),~~
A_{12}=\left(\hspace{-1mm}
 \begin{array}{ccc}
 0 & 1 & 0\\
 0 & 3 & 0\\
 0&  1 & 0\\
 \end{array}
\hspace{-1mm}\right),
\end{align*}
\begin{align*}
A_{13}=\left(\hspace{-1mm}
 \begin{array}{ccc}
 0 & 0 & 1\\
 0 & 3 & 0\\
 1&  0 & 0\\
 \end{array}
\hspace{-1mm}\right),~~
A_{14}=\left(\hspace{-1mm}
 \begin{array}{ccc}
 0 & 0 & 0\\
 1 & 3 & 1\\
 0&  0 & 0\\
 \end{array}
\hspace{-1mm}\right),~~
A_{15}=\left(\hspace{-1mm}
 \begin{array}{ccc}
 0 & 0 & 0\\
 0 & 5 & 0\\
 0&  0 & 0\\
 \end{array}
\hspace{-1mm}\right).
\end{align*}

From the definition of $y_{A_{i}}$, we can obtain
\begin{align*}
y_{A_1}=&s_1,\quad y_{A_2}=y_{A_3}=y_{A_4}=y_{A_8}=y_{A_{10}}=y_{A_{12}}=y_{A_{13}}=y_{A_{14}}=y_{A_{15}}=s_{0}s_1s_{0}s_1,\\
&y_{A_5}=s_1s_{0},\quad y_{A_6}=s_1s_{0}s_1, \quad y_{A_7}=s_{0}s_1, \quad y_{A_9}=s_0s_{1}s_0,\quad y_{A_{11}}=s_{0}.
\end{align*}
We have $\Lambda^{\jmath}(3, 2)=\big\{\lambda_1=(0,5,0), \lambda_2=(1,3,1), \lambda_3=(2,1,2)\big\}$; correspondingly, we have $W_{\lambda_1}=W_{B_d}$, $W_{\lambda_2}=\{e, s_{0}\}$ and $W_{\lambda_3}=\{e, s_{1}\}$. By \cite[Lemma 2.2.1]{LL21}, we can compute each $w_{A_{i}}$ for $1\leq i\leq 15$. Therefore, we obtain the longest element $w_{A_{i}}^{+}$ in each double coset $W_{ro(A_{i})} w_{A_{i}} W_{co(A_{i})}$ corresponding to $A_{i}$. By a direct calculation, we can verify that for each $i$, $y_{A_{i}}=w_{A_{i}}^{+}$, and moreover, $\ell(w_{A_{i}}^{+})$ coincides with the integer calculated via \eqref{b3}. We shall present the case $i=5$ and leave the remaining ones to the reader. We have $ro(A_{5})=\lambda_3$, $co(A_{5})=\lambda_2$ and $w_{A_{5}}=e$. Therefore, we obtain $w_{A_{5}}^{+}=s_1s_{0}=y_{A_{5}}$. Moreover,
\begin{align*}
2^{2}-\frac{1}{2}\big(a_{00}^{\natural}(1+1)+a_{10}^{\natural}(1+0)+a_{11}^{\natural}(1+1+0+1)\big)=2=\ell(w_{A_{5}}^{+}).
\end{align*}
\end{example}

The next lemma easily follows from \cite[\S3.5, Proof of Theorem 2.1]{Du96}.
\begin{lem}\label{partition-longest element}
For each $A=(ro(A), w_{A}^{+}, co(A))\in \Pi_{n, d}$, we have $\sigma(A)=PT(w_{A}^{+})$.
\end{lem}
\begin{proof}
From Lemma \ref{longest element} and the proof of \cite[\S3.5, Theorem 2.1]{Du96}, we see that the number $\mathfrak{s}_{k}(A)$ defined above is the maximal cardinality of a subset of $[-d, d]$ which is a disjoint union of $k$ subsets each of which has its natural order reversed by $w_{A}^{+}$. Thus, we see that the partition $\sigma(A)$ is exactly the partition associated to $w_{A}^{+}$ as defined in \cite[\S7]{Lu85b}, recalling that $w_{A}^{+}$ is considered as an element of $\mathfrak{S}_{2d+1}$. From this and \cite[Chapter 21]{Shi86}, we have $\sigma(A)=PT(w_{A}^{+})$.
\end{proof}

\begin{example}
Let
\begin{align*}
A=\left(\hspace{-1mm}
 \begin{array}{ccc}
 2 & 1 & 1\\
 2 & 3 & 2\\
 1&  1 & 2\\
 \end{array}
\hspace{-1mm}\right)\in \Pi_{3, 7}.
\end{align*}
Then $$A_{+}=\left(\hspace{-1.5mm}
 \begin{array}{ccc}
 (-4,-5) & (-6) & (-7)\\
 (3,2) & (1,0,-1) & (-2,-3)\\
 (7)&  (6) & (5,4)\\
 \end{array}
\hspace{-1.5mm}\right),$$ and the associated $y_{A}$ is
\begin{equation*}  \label{dim:Sirrep1-waplus}
y_{A}=\left(\hspace{-1.6mm}
 \begin{array}{ccccccccccccccccccccccccc}
 -7&-6&-5&-4&-3&-2&-1&0&1&2&3&4&5&6&7\\
 7&3&2&-4&-5&6&1&0&-1&-6&5&4&-2&-3&-7\\
\end{array}
\hspace{-1.6mm}\right).
\end{equation*}
Thus, we have $y_{A}=w_{A}^{+}$ by Lemma \ref{longest element}. We can compute $\mathfrak{s}_{1}(A)=1+2+3+2+1=9$, $\mathfrak{s}_{2}(A)=1+2+2+1+1+1+3+2=13$, $\mathfrak{s}_{3}(A)=15$. By Theorem \ref{theor:d=d2=d000=du-cells}, we have $\sigma(A)=(9,4,2)$; on the other hand, using the Robinson--Schensted algorithm we deduce that $PT(w_{A}^{+})=PT(y_{A})$ also equals $(9,4,2)$.

By a similar calculation, for each $A_{i}$ ($1\leq i\leq 15$) in Example \ref{b4}, one can obtain $\sigma(A_{i})=PT(w_{A_{i}}^{+})$. We leave it as an exercise to the reader.
\end{example}

Combining Propositions \ref{lem:d=d2=d000ad} and \ref{theor:d=d2=d000adadc} with Lemma \ref{partition-longest element}, we can get the following theorem.
\begin{thm}
\label{theor:charac of cells in Schur algs}
For any two elements $A, A'\in \Pi_{n, d}$, we have $\{A\}\sim_{LR} \{A'\}$ if and only if $\sigma(A)\approx \sigma(A')$.
\end{thm}
\begin{proof}
By Proposition \ref{lem:d=d2=d000ad}(3), we see $\{A\}\sim_{LR} \{A'\}$ if and only if $w_{A}^{+}\sim_{LR} w_{A'}^{+}$. By Proposition \ref{theor:d=d2=d000adadc}, we have $w_{A}^{+}\sim_{LR} w_{A'}^{+}$ if and only if $PT(w_{A}^{+})\approx PT(w_{A'}^{+})$. By Lemma \ref{partition-longest element}, we know $\sigma(A)=PT(w_{A}^{+})$ and $\sigma(A')=PT(w_{A'}^{+})$. Thus, $\{A\}\sim_{LR} \{A'\}$ if and only if $\sigma(A)\approx \sigma(A')$.
\end{proof}

\begin{rem}\label{remark:onetoone correspondence}
Let $\mathcal{P}_{d}^{n}$ be the subset of $\mathcal{P}_{d}$ consisting of partitions in $\mathcal{P}_{d}$ with at most $n$ parts. By Theorem \ref{theor:d=d2=d000=du-cells}, we have a map $\sigma :A\mapsto \sigma(A)=(\sigma_{1}(A), \sigma_{2}(A), \ldots, \sigma_{n}(A))$, which is from $\Pi_{n, d}$ to the set of partitions of $2d+1$ with at most $n$ parts. We can further show that $\sigma^{-1}(\lambda)$ is nonempty for each $\lambda\in \mathcal{P}_{d}^{n}$. From this and Theorem \ref{theor:charac of cells in Schur algs}, we see that there is a one-to-one correspondence between the set of two-sided cells in $S^{\jmath}(n, d)$ and special partitions of $2d+1$ with at most $n$ parts.
\end{rem}

\begin{example}\label{d1}
Let us return to Example \ref{b4}. By Theorem \ref{theor:d=d2=d000=du-cells}, we have
\begin{align*}
\sigma(A_{1})&=(2,2,1),\quad \sigma(A_5)=\sigma(A_6)=\sigma(A_7)=\sigma(A_{11})=(3,1,1),\quad \sigma(A_9)=(3,2),\\
\sigma(A_2)=&\sigma(A_3)=\sigma(A_4)=\sigma(A_8)=\sigma(A_{10})=\sigma(A_{12})=\sigma(A_{13})=\sigma(A_{14})=\sigma(A_{15})=(5).
\end{align*}

Since $(2,2,1)=par[\Lambda_{1}]$, $(3,1,1)=par[\Lambda_{2}]$, $(3,2)=par[\Lambda_{3}]$, where $\Lambda_{1}={0<1 \choose 2}$, $\Lambda_{2}={0<2 \choose 1}$, $\Lambda_{3}={1<2 \choose 0}$, and $\{0, 1, 2\}, \{0, 2, 1\}, \{1, 2, 0\}$ are mutually permuted, we have $[\Lambda_{1}]\approx [\Lambda_{2}]\approx [\Lambda_{3}]$, and hence $(2,2,1)\approx(3,1,1)\approx(3,2)$. Thus, by Theorem \ref{theor:charac of cells in Schur algs}, $\{A_{1}\},  \{A_5\}, \{A_6\}, \{A_7\}, \{A_9\}, \{A_{11}\}$ lie in the same two-sided cell $\mathfrak{C}_1$.

Since $(5)=par[\Lambda_{4}]$, where $\Lambda_{4}={0<3 \choose 0}$, and $\{0, 3, 0\}$ is not a permutation of $\{0, 1, 2\}$, we have $[\Lambda_{4}]\not\approx[\Lambda_{1}]$, and hence $(5)\not\approx (2,2,1)$. Therefore, we see that $\{A_2\}, \{A_3\}, \{A_4\}, \{A_8\}, \{A_{10}\}, \{A_{12}\}, \{A_{13}\}, \{A_{14}\}, \{A_{15}\}$ lie in a two-sided cell $\mathfrak{C}_2$ different from $\mathfrak{C}_1$.

It is easy to see that there are exactly three special partitions of $5$: $(5)$, $(3,1,1)$, $(1,1,1,1,1)$, and the special partitions of $5$ with at most $3$ parts are $(5)$ and $(3,1,1)$. Therefore, there is a bijection between the set of two-sided cells in $S^{\jmath}(3, 2)$ and special partitions of $5$ with at most $3$ parts.
\end{example}

\subsection{Number of left cells in a two-sided cell of $S^{\jmath}(n, d)$}
For each $\lambda\in \mathcal{SP}_{d}$, we denote by $\mathcal{O}_{\lambda}$ the corresponding {\em special} nilpotent orbit, and $\mathbf{c}_{\lambda}$ the associated two-sided cell in $\HH$. We have the Springer representation Sp$(\mathcal{O}_{\lambda})$ of $W_{B_d}$ attached to $\mathcal{O}_{\lambda}$ (and the trivial representation of its fundamental group). Then $\{\mathrm{Sp}(\mathcal{O}_{\lambda})\:|\:\lambda\in \mathcal{SP}_{d}\}$ forms a complete set of special representations of $W_{B_d}$ in the sense of Lusztig (see \cite{Lu79} and also \cite{Lu84}). Moreover, the number of left cells in $\mathbf{c}_{\lambda}$ equals the dimension of Sp$(\mathcal{O}_{\lambda})$. On the other hand, the number of orbital varieties (following Joseph) contained in $\mathcal{O}_{\lambda}$ equals the dimension of Sp$(\mathcal{O}_{\lambda})$. By \cite[Theorem 1]{Mc99}, this number is also equal to the number of standard domino tableaux of shape $\lambda$. Thus, we see that for any $\lambda\in \mathcal{SP}_{d}$, the number of left cells in $\mathbf{c}_{\lambda}$ equals that of standard domino tableaux of shape $\lambda$.

We denote by $\mathcal{SP}_{d}^{n}$ the set of special partitions of $2d+1$ with at most $n$ parts. For each $\lambda\in \mathcal{SP}_{d}^{n}$, we denote by $\mathbf{c}_{\lambda}^{\jmath}$ the associated two-sided cell in $S^{\jmath}(n, d)$ by Remark \ref{remark:onetoone correspondence}. For any $w\in W_{B_d}$, let $\mathcal{R}(w)=\{s\in S\:|\:ws<w\}$. It is known that $\mathcal{R}(x)=\mathcal{R}(y)$ if $x\sim_{L} y$; for $\Gamma$ a left cell in $\HH$, we shall write $\mathcal{R}(\Gamma)$ for the set $\mathcal{R}(w)$, where $w$ is any element of $\Gamma$. The number of left cells in $S^{\jmath}(n, d)$ which correspond to $\Gamma$ is the number of $\lambda\in \Lambda^{\jmath}(n, d)$ for which the simple reflections of $W_{\lambda}$ lie in $\mathcal{R}(\Gamma)$ by Proposition \ref{lem:d=d2=d000ad}. The following gives a conjecture on the number of left cells in a two-sided cell of $S^{\jmath}(n, d)$.
\begin{conj}
\label{theor:d=d2=d000adadcnumber}
The number of left cells in $\mathbf{c}_{\lambda}^{\jmath}$ equals the number of semistandard domino tableaux of shape $\lambda$ with all entries in the dominoes $\leq r+1$ and the entry in the monomino $1$.
\end{conj}
\begin{rem}
The latter number in Conjecture \ref{theor:d=d2=d000adadcnumber} has appeared in \cite{BK00} and also \cite{St96}. We refer to \cite[\S3]{St96} for further explanation.
\end{rem}


\section{Characterization of two-sided cells in modified $\imath$quantum groups of type AIII}
In this section, we preserve the setup of Section \ref{c5}. Inspired by \cite{Du96} and \cite{McG03} we shall lift the combinatorial description of two-sided cells in $S^{\jmath}(n, d)$ to give a characterization of two-sided cells in the modified $\imath$quantum group $\dot{\mathbf{U}}^{\jmath}(\mathfrak{s}\mathfrak{l}_{n})$.

\subsection{Definition of $\dot{\mathbf{U}}^{\jmath}(\mathfrak{s}\mathfrak{l}_{n})$}
\label{c2}
In this subsection, we shall follow \cite[\S2]{BSWW18} and \cite[\S6.1]{BW18} to give some preliminaries on the modified $\imath$quantum group $\dot{\mathbf{U}}^{\jmath}(\mathfrak{s}\mathfrak{l}_{n})$.

For $a\in \mathbb{Z}$ and $b\in \mathbb{N}$, we define
$$[a]=\frac{q^{a}-q^{-a}}{q-q^{-1}},\quad [b]!=\prod_{h=1}^{b}\frac{q^{h}-q^{-h}}{q-q^{-1}}.$$
The $b$-th {\em divided power} of an element $E$ in a $\mathbb{Q}(q)$-algebra is the element $E^{(b)} :=E^{b}/[b]!$.

Recall that $n=2r+1\in \mathbb{Z}_{\geq 3}$. We set $\diamond=\frac{1}{2}$ and define
\begin{align*}
\mathbb{I}=\{i\in \mathbb{Z}+\diamond\:|\:-r< i< r\}.
\end{align*}
Consider the root datum of type $A_{n-1}$ with a Cartan matrix indexed by $\mathbb{I}$, weight lattice $X$, simple roots $\{\alpha_{i}\}_{i\in \mathbb{I}}\subset X$, simple coroots $\{\alpha_{i}^{\vee}\}_{i\in \mathbb{I}}$, and the coroot lattice $Y=\bigoplus_{i\in \mathbb{I}}\mathbb{Z}\alpha_{i}^{\vee}$. There is a perfect pairing
\begin{align}\label{c1}
\langle \cdot, \cdot\rangle :Y\times X\rightarrow \mathbb{Z}
\end{align}
such that the entries of the Cartan matrix are given by $\langle \alpha_{i}^{\vee}, \alpha_{j}\rangle$ for $i, j\in \mathbb{I}$.

Consider the lattice $\bigoplus_{a=-r}^{r}\mathbb{Z}\varepsilon_{a}$ with the standard bilinear pairing $(\varepsilon_{a}, \varepsilon_{a})=\delta_{a, b}$ for all $a, b$. We will identify
\begin{align*}
X=\bigoplus_{a=-r}^{r}\mathbb{Z}\varepsilon_{a}\Big/X_{\varepsilon},\quad\text{ where } X_{\varepsilon}=\mathbb{Z}\Big(\sum_{a=-r}^{r}\varepsilon_{a}\Big).
\end{align*}
Then $\alpha_{i}=\varepsilon_{i-\diamond}-\varepsilon_{i+\diamond}$ (mod $X_{\varepsilon}$) lies in $X$.

The quantum group $\mathbf{U}_{q}(\mathfrak{s}\mathfrak{l}_{n})$ of type $A_{n-1}$ is an associative algebra over $\mathbb{Q}(q)$ with generators
$E_{i}, F_{i}, K_{i}^{\pm 1}$ $(i\in \mathbb{I})$ and the following relations:
\begin{align*}&K_{i}K^{\pm1}_{j}=K^{\pm1}_{j}K_{i}, \quad K_{i}K_{i}^{-1}=1=K_{i}^{-1}K_{i},\quad K_{i}E_{j}K_{i}^{-1}=q^{\langle \alpha_{i}^{\vee}, \alpha_{j}\rangle}E_{j}\
~\quad\hbox{for~all}~ i, j\in \mathbb{I},\\
&E_{i}F_{j}-F_{j}E_{i}=\delta_{i,j}\frac{K_{i}-K_{i}^{-1}}{q-q^{-1}},\quad K_{i}F_{j}K_{i}^{-1}=q^{-\langle \alpha_{i}^{\vee}, \alpha_{j}\rangle}F_{j}
\ ~\quad\hbox{for~all}~ i, j\in \mathbb{I},\\
&\sum\limits_{s+t=1-\langle \alpha_{i}^{\vee}, \alpha_{j}\rangle}(-1)^{t}E_{i}^{(s)}E_{j}E_{i}^{(t)}=\sum\limits_{s+t=1-\langle \alpha_{i}^{\vee}, \alpha_{j}\rangle}(-1)^{t}F_{i}^{(s)}F_{j}F_{i}^{(t)}=0\  ~\quad\hbox{for~all }i\neq j.
\end{align*}

We set
\begin{align*}
\mathbb{I}^{\jmath}=\{\diamond, \diamond+1,\ldots,\diamond+r-1\}.
\end{align*}
By \cite[\S6.1]{BW18}, the $\imath$quantum group $\mathbf{U}^{\jmath}(\mathfrak{s}\mathfrak{l}_{n})$ is an associative algebra over $\mathbb{Q}(q)$ with generators
$\mathcal{E}_{i}, \mathcal{F}_{i}, \mathcal{K}_{i}^{\pm 1}$ $(i\in \mathbb{I}^{\jmath})$ and the following relations:
\begin{align*}
\mathcal{K}_{i}\mathcal{K}^{\pm1}_{j}&=\mathcal{K}^{\pm1}_{j}\mathcal{K}_{i}, \quad \mathcal{K}_{i}\mathcal{K}_{i}^{-1}=1=\mathcal{K}_{i}^{-1}\mathcal{K}_{i},\\
\mathcal{K}_{i}\mathcal{E}_{j}\mathcal{K}_{i}^{-1}&=q^{(\alpha_{i}-\alpha_{-i}, \alpha_{j})}\mathcal{E}_{j},\quad \mathcal{K}_{i}\mathcal{F}_{j}\mathcal{K}_{i}^{-1}=q^{-(\alpha_{i}-\alpha_{-i}, \alpha_{j})}\mathcal{F}_{j},\\
\mathcal{E}_{i}\mathcal{F}_{j}-\mathcal{F}_{j}\mathcal{E}_{i}&=0\quad \text{ for all $i\neq j$}, \quad \mathcal{E}_{i}\mathcal{F}_{i}-\mathcal{F}_{i}\mathcal{E}_{i}=\frac{\mathcal{K}_{i}-\mathcal{K}_{i}^{-1}}{q-q^{-1}}\quad \text{ for all }i\neq \diamond,\\
\mathcal{E}_{i}\mathcal{E}_{j}&=\mathcal{E}_{j}\mathcal{E}_{i},\quad \mathcal{F}_{i}\mathcal{F}_{j}=\mathcal{F}_{j}\mathcal{F}_{i}\quad \text{ for }|i-j|> 1,\\
\mathcal{E}_{i}^{2}\mathcal{E}_{j}+\mathcal{E}_{j}\mathcal{E}_{i}^{2}&=(q+q^{-1})\mathcal{E}_{i}\mathcal{E}_{j}\mathcal{E}_{i},\quad \mathcal{F}_{i}^{2}\mathcal{F}_{j}+\mathcal{F}_{j}\mathcal{F}_{i}^{2}=(q+q^{-1})\mathcal{F}_{i}\mathcal{F}_{j}\mathcal{F}_{i}\quad \text{ for }|i-j|= 1,\\
\mathcal{F}_{\diamond}^{2}\mathcal{E}_{\diamond}+\mathcal{E}_{\diamond}\mathcal{F}_{\diamond}^{2}&=(q+q^{-1})
(\mathcal{F}_{\diamond}\mathcal{E}_{\diamond}\mathcal{F}_{\diamond}-q\mathcal{F}_{\diamond}\mathcal{K}_{\diamond}^{-1}-q^{-1}\mathcal{F}_{\diamond}\mathcal{K}_{\diamond}),\\
\mathcal{E}_{\diamond}^{2}\mathcal{F}_{\diamond}+\mathcal{F}_{\diamond}\mathcal{E}_{\diamond}^{2}&=(q+q^{-1})
(\mathcal{E}_{\diamond}\mathcal{F}_{\diamond}\mathcal{E}_{\diamond}-q^{-1}\mathcal{K}_{\diamond}\mathcal{E}_{\diamond}-q\mathcal{K}_{\diamond}^{-1}\mathcal{E}_{\diamond}).
\end{align*}

By \cite[Proposition 6.2]{BW18}, There is an injective $\mathbb{Q}(q)$-algebra homomorphism $\jmath :\mathbf{U}^{\jmath}(\mathfrak{s}\mathfrak{l}_{n})\rightarrow \mathbf{U}_{q}(\mathfrak{s}\mathfrak{l}_{n})$, which is defined by
\begin{align*}
\mathcal{E}_{i}\mapsto E_{i}+K_{i}^{-1}F_{-i},\quad \mathcal{F}_{i}\mapsto F_{i}K_{-i}^{-1}+E_{-i},\quad \mathcal{K}_{i}\mapsto K_{i}K_{-i}^{-1}
\end{align*}
for all $i\in \mathbb{I}^{\jmath}$, such that $\mathbf{U}^{\jmath}(\mathfrak{s}\mathfrak{l}_{n})$ is a (right) {\em coideal} subalgebra of $\mathbf{U}_{q}(\mathfrak{s}\mathfrak{l}_{n})$ (cf. \cite[Proposition 6.3]{BW18}). Then $(\mathbf{U}_{q}(\mathfrak{s}\mathfrak{l}_{n}), \mathbf{U}^{\jmath}(\mathfrak{s}\mathfrak{l}_{n}))$ is an example of quantum symmetric pairs (cf. \cite{Le99}, \cite{Ko14}).

Let $\theta$ be the involution of the lattice $X$ defined by $\theta(\varepsilon_{a})=-\varepsilon_{-a}$ for $-r\leq a\leq r$. Denote by $X^{\theta}$ the sublattice of $\theta$-fixed points in $X$. Note that $\theta(\alpha_{i})=\alpha_{-i}$ for all $i\in \mathbb{I}$; hence $\theta$ induces an automorphism of the root system. Set ${}^{\theta}\!\alpha_{i}^{\vee}=\alpha_{i}^{\vee}-\alpha_{-i}^{\vee}$ for $i\in \mathbb{I}^{\jmath}$. Let
\begin{align*}
X_{\jmath}=X/X^{\theta},\qquad Y^{\jmath}=\bigoplus_{i\in \mathbb{I}^{\jmath}}\mathbb{Z}\hspace{0.6mm} {}^{\theta}\!\alpha_{i}^{\vee}.
\end{align*}
The lattice $X_{\jmath}$ can be regarded as a weight lattice for $\mathbf{U}^{\jmath}(\mathfrak{s}\mathfrak{l}_{n})$. The pairing \eqref{c1} induces a non-degenerate pairing
\begin{align*}
\langle \cdot, \cdot\rangle :Y^{\jmath}\times X_{\jmath}\rightarrow \mathbb{Z}.
\end{align*}
For $\lambda\in X_{\jmath}$, we write
\begin{align*}
\lambda_{i}=\langle {}^{\theta}\!\alpha_{i}^{\vee}, \lambda\rangle \quad \text{ for all }i\in \mathbb{I}^{\jmath}.
\end{align*}

Let $\dot{\mathbf{U}}^{\jmath}(\mathfrak{s}\mathfrak{l}_{n})$ be the $\mathbb{Q}(q)$-linear category with the object set $X_{\jmath}$ and morphisms generated by $\mathcal{E}_{i} :\lambda\mapsto \lambda+\alpha_{i}=\lambda-\alpha_{-i}$, $\mathcal{F}_{i} :\lambda\mapsto \lambda-\alpha_{i}=\lambda+\alpha_{-i}$, for all $i\in \mathbb{I}^{\jmath}$, subject to the following relations:
\begin{align*}
(\mathcal{E}_{i}\mathcal{F}_{j}-\mathcal{F}_{j}\mathcal{E}_{i})1_{\lambda}&=0\quad \text{ for all $i\neq j$},\\
(\mathcal{E}_{i}\mathcal{F}_{i}-\mathcal{F}_{i}\mathcal{E}_{i})1_{\lambda}&=[\lambda_{i}]1_{\lambda}\quad \text{ for all }i\neq \diamond,\\
\sum\limits_{s+t=1-\langle \alpha_{i}^{\vee}, \alpha_{j}\rangle}(-1)^{t}\mathcal{E}_{i}^{(s)}\mathcal{E}_{j}\mathcal{E}_{i}^{(t)}&=\sum\limits_{s+t=1-\langle \alpha_{i}^{\vee}, \alpha_{j}\rangle}(-1)^{t}\mathcal{F}_{i}^{(s)}\mathcal{F}_{j}\mathcal{F}_{i}^{(t)}=0\  ~\quad\hbox{for~all }i\neq j,\\
(\mathcal{E}_{\diamond}^{(2)}\mathcal{F}_{\diamond}-\mathcal{E}_{\diamond}\mathcal{F}_{\diamond}\mathcal{E}_{\diamond}+\mathcal{F}_{\diamond}\mathcal{E}_{\diamond}^{(2)})1_{\lambda}&=
-(q^{\lambda_{\diamond}+2}+q^{-\lambda_{\diamond}-2})\mathcal{E}_{\diamond}1_{\lambda},\\
(\mathcal{F}_{\diamond}^{(2)}\mathcal{E}_{\diamond}-\mathcal{F}_{\diamond}\mathcal{E}_{\diamond}\mathcal{F}_{\diamond}+\mathcal{E}_{\diamond}\mathcal{F}_{\diamond}^{(2)})1_{\lambda}&=
-(q^{\lambda_{\diamond}-1}+q^{-\lambda_{\diamond}+1})\mathcal{F}_{\diamond}1_{\lambda},
\end{align*}
where $1_{\lambda}$ denotes the identity endomorphism of $\lambda$.

As mentioned in \cite[p.651]{BSWW18}, the definition of $\dot{\mathbf{U}}^{\jmath}(\mathfrak{s}\mathfrak{l}_{n})$ here is essentially the same as the one in \cite[\S4.1]{LiW18}. Let ${}_{\A}\!\dot{\mathbf{U}}^{\jmath}$ be the $\A$-linear subcategory of $\dot{\mathbf{U}}^{\jmath}(\mathfrak{s}\mathfrak{l}_{n})$ with the same objects and with morphisms generated by divided powers $\mathcal{E}_{i}^{(a)}1_{\lambda}$, $\mathcal{F}_{i}^{(a)}1_{\lambda}$ for all $i\in \mathbb{I}^{\jmath}$, $\lambda\in X_{\jmath}$ and $a\in \mathbb{Z}_{\geq 0}$. It was shown in \cite[\S4.1]{LiW18} that ${}_{\A}\!\dot{\mathbf{U}}^{\jmath}$ is a free $\A$-module and $\mathbb{Q}(q)\otimes_{\A} {}_{\A}\!\dot{\mathbf{U}}^{\jmath}=\dot{\mathbf{U}}^{\jmath}(\mathfrak{s}\mathfrak{l}_{n})$.

\subsection{Description of two-sided cells in ${}_{\A}\!\dot{\mathbf{U}}^{\jmath}$}
\label{c3}

In \cite[Theorem 5.5]{LiW18}, an $\A$-basis $\dot{\mathbf{B}}^{\jmath}(\mathfrak{s}\mathfrak{l}_{n})$ of ${}_{\A}\!\dot{\mathbf{U}}^{\jmath}$, which is called the {\em $\jmath$-canonical basis}, has been constructed via studying the asymptotical behavior of the canonical bases of $\jmath$-Schur algebras $S^{\jmath}(n, d)$ with varying $d$ under the transfer map. Furthermore, in \cite[Theorem 5.6]{LiW18} it was proved that the structure constants with respect to the $\jmath$-canonical basis $\dot{\mathbf{B}}^{\jmath}(\mathfrak{s}\mathfrak{l}_{n})$ lie in $\mathbb{N}[q, q^{-1}]$.



In \cite[(4.12)]{LiW18}, a surjective algebra homomorphism $\phi_{d}^{\jmath} :{}_{\A}\!\dot{\mathbf{U}}^{\jmath}\rightarrow S^{\jmath}(n, d)$ has been constructed for each $d$. The following theorem was proved in \cite[Corollary 6.6]{BSWW18}, improving \cite[Proposition 5.11]{LiW18}.

\begin{thm} $($see \cite[Corollary 6.6]{BSWW18}$)$
\label{good properties of the based map}
For each $b\in \dot{\mathbf{B}}^{\jmath}(\mathfrak{s}\mathfrak{l}_{n})$, the image $\phi_{d}^{\jmath}(b) \in \{0\}\cup \mathbf{B}_{d}^{\jmath}$, where $\mathbf{B}_{d}^{\jmath}$ denotes the set of canonical basis elements for $S^{\jmath}(n, d)$. Moreover, the kernel of $\phi_{d}^{\jmath}$ is spanned by the elements $b\in \dot{\mathbf{B}}^{\jmath}(\mathfrak{s}\mathfrak{l}_{n})$ such that $\phi_{d}^{\jmath}(b)=0$.
\end{thm}

In \cite[(4.8)]{LiW18}, the transfer map $\phi_{d+n, d}^{\jmath} :S^{\jmath}(n, d+n)\rightarrow S^{\jmath}(n, d)$ has been defined. Moreover, in \cite[Proposition 4.4]{LiW18} it was proved that the maps $\phi_{d}^{\jmath}$ and $\phi_{d+n, d}^{\jmath}$ are compatible, that is, $\phi_{d+n, d}^{\jmath}\circ\phi_{d+n}^{\jmath}=\phi_{d}^{\jmath}$.

Set $\hat{S}^{\jmath} :=\lim\limits_{\longleftarrow}S^{\jmath}(n, d)$, where the limit is taken over the projective system given by the transfer maps $(\phi_{d+n, d}^{\jmath})_{d\in \mathbb{N}}$ mentioned above. Since the maps $\phi_{d}^{\jmath}$ are compatible with this system, there is a unique map $\phi :{}_{\A}\!\dot{\mathbf{U}}^{j}\rightarrow \hat{S}^{\jmath}$, which factors each of the maps $\phi_{d}^{\jmath}$ through the canonical map $\hat{S}^{\jmath}\rightarrow S^{\jmath}(n, d)$.

Recall that a bilinear form $\langle\cdot, \cdot\rangle_{d}$ on $S^{\jmath}(n, d)$ was defined in \cite[\S3.7]{BKLW18} (and denoted by $(\cdot, \cdot)_{D}$ therein with $D=2d+1$). Using this, in \cite[(5.1)]{LiW18}, a bilinear form $\langle\cdot, \cdot\rangle_{\jmath}$ on ${}_{\A}\!\dot{\mathbf{U}}^{\jmath}$ has been constructed as follows:
\begin{align}\label{c4}
\langle x, y\rangle_{\jmath} :=\lim_{p\rightarrow \infty} \sum_{k=0}^{n-1}\big\langle \phi_{k+pn}^{\jmath}(x), \phi_{k+pn}^{\jmath}(y)\big\rangle_{k+pn}\quad \text{ for }x, y\in {}_{\A}\!\dot{\mathbf{U}}^{\jmath}.
\end{align}
Moreover, it was shown in \cite[Proposition 5.7]{LiW18} that the bilinear form $\langle\cdot, \cdot\rangle_{\jmath}$ is non-degenerate. Using this and an argument similar to \cite[Proposition 7.1]{McG12}, we can prove the following result. (We also refer to \cite[Proposition 6.5.1]{FLLLW20} for a similar conclusion in the case of affine $\mathfrak{s}\mathfrak{l}$ type.)
\begin{prop}
\label{injective-mapphi}
The homomorphism $\phi$ is injective.
\end{prop}
\begin{proof}
Suppose that $u$ is in the kernel of $\phi$. Then for each $d$ we have $\phi_{d}^{\jmath}(u)=0$, and hence by \eqref{c4} $u$ is in the radical of the bilinear form $\langle\cdot, \cdot\rangle_{\jmath}$. Since it is non-degenerate, we must have $u=0$.
\end{proof}

The following theorem gives a characterization of two-sided cells in ${}_{\A}\!\dot{\mathbf{U}}^{\jmath}$ in terms of the ones in $\jmath$-Schur algebras $S^{\jmath}(n, d)$.
\begin{thm}\label{characte-two-sided-cells-j}
For any two elements $b, b'\in \dot{\mathbf{B}}^{\jmath}(\mathfrak{s}\mathfrak{l}_{n})$, we have $b\sim_{LR}b'$ if and only if there exists some $d\in \mathbb{N}$ such that $\phi_{d}^{\jmath}(b)\neq 0$, $\phi_{d}^{\jmath}(b')\neq 0$, and moreover, $\phi_{d}^{\jmath}(b)\sim_{LR}\phi_{d}^{\jmath}(b')$.
\end{thm}
\begin{proof}
Assume that $b\sim_{LR}b'$. By Proposition \ref{injective-mapphi}, we see that there exists some $d\in \mathbb{N}$ such that $\phi_{d}^{\jmath}(b')\neq 0$. By Lemma \ref{another-charact-cells}, there exist $h, h'\in \dot{\mathbf{B}}^{\jmath}(\mathfrak{s}\mathfrak{l}_{n})$ such that $b'$ appears with nonzero coefficient in the expansion of $hbh'$ as a linear combination of $\jmath$-canonical basis elements. By Theorem \ref{good properties of the based map}, we see that the canonical basis element $\phi_{d}^{\jmath}(b')$ appears with nonzero coefficient when expanding $\phi_{d}^{\jmath}(h)\phi_{d}^{\jmath}(b)\phi_{d}^{\jmath}(h')$ as a linear combination of elements in $\mathbf{B}_{d}^{\jmath}$. Hence, $\phi_{d}^{\jmath}(b)\neq 0$ and $\phi_{d}^{\jmath}(b')\preceq_{LR} \phi_{d}^{\jmath}(b)$. Similarly, we can show that $\phi_{d}^{\jmath}(b)\preceq_{LR} \phi_{d}^{\jmath}(b')$. Thus, we have $\phi_{d}^{\jmath}(b)\sim_{LR}\phi_{d}^{\jmath}(b')$.

Conversely, assume that there exists some $d\in \mathbb{N}$ such that $\phi_{d}^{\jmath}(b)\neq 0$, $\phi_{d}^{\jmath}(b')\neq 0$ and $\phi_{d}^{\jmath}(b)\sim_{LR}\phi_{d}^{\jmath}(b')$. By Lemma \ref{another-charact-cells} and the surjectivity of $\phi_{d}^{\jmath}$, we see that there exist $h, h'\in {}_{\A}\!\dot{\mathbf{U}}^{\jmath}$ such that $\phi_{d}^{\jmath}(b')$ appears with nonzero coefficient in the expansion of $\phi_{d}^{\jmath}(h)\phi_{d}^{\jmath}(b)\phi_{d}^{\jmath}(h')$ as a linear combination of elements in $\mathbf{B}_{d}^{\jmath}$. By Theorem \ref{good properties of the based map}, we see that $b'$ appears with nonzero coefficient in the expansion of $hbh'$ as a linear combination of $\jmath$-canonical basis elements. Hence, $b'\preceq_{LR} b$. Similarly, we can show that $b\preceq_{LR}b'$. Thus, we have $b\sim_{LR}b'$.
\end{proof}


\begin{rem}\label{affine-type-based module maps}
In \cite{FLLLW20, FLLLW22}, the authors have constructed several types of affine $q$-Schur algebras as well as their canonical bases, and shown that there exist maps from the modified $\imath$quantum groups of affine $\mathfrak{s}\mathfrak{l}$ type to them. If we can prove that some of these maps satisfy the properties similar to the ones that $\phi_{d}^{\jmath}$ satisfies in Theorem \ref{good properties of the based map}, and that an argument similar to Proposition \ref{injective-mapphi} holds (cf. \cite[Proposition 6.5.1]{FLLLW20}), then we can give an assertion similar to Theorem \ref{characte-two-sided-cells-j}.

In \cite[\S5.3]{CLW20}, we have taken the first step towards solving this problem, that is, we have provided a classification of two-sided cells in the affine $q$-Schur algebras of arbitrary type in terms of the ones in the corresponding affine Hecke algebras.
\end{rem}

\section{Description of two-sided cells in $S^{\imath}(n, d)$}
In this section, we focus on the $\imath$-Schur algebra $S^{\imath}(n, d)$ associated to $\HH$ and formulate a combinatorial description of two-sided cells in it. Moreover, we state a conjecture on the number of left cells in a two-sided cell of it.

\subsection{Preliminaries}
\label{e2}
Fix $d\in \Z_{\geq 1}$. Let $W_{C_d}$ be the Weyl group of type $C_d$ with generators $S=\{s_{0}, s_{1}, \ldots, s_{d-1}\}$ and with the same relations as in $W_{B_d}$. There is a canonical isomorphism between $W_{C_d}$ and $W_{B_d}$, and we shall denote the Iwahori--Hecke algebra associated to $W_{C_d}$ also by $\HH$. Then $\HH$ is an algebra over $\A$ with a basis $\{T_{w}\:|\:w\in W_{C_d}\}$ satisfying the same relations as \eqref{Btype-Hecke-alg-rela}.



Fix $n=2r\in \mathbb{Z}_{\geq 2}$. Following \cite[\S6.1]{LL21} we set
$$\Lambda^{\imath}(n, d)=\{\lambda=(\lambda_{i})_{i\in [-r, r]\setminus \{0\}}\in \mathbb{N}^{n}\:|\:\lambda_{-i}=\lambda_{i}\mathrm{~for~}1\leq i\leq r, ~\sum_{i}\lambda_{i}=2d\}.$$
For each $\lambda\in \Lambda^{\imath}(n, d)$, we denote by $W_{\lambda}$ the parabolic subgroup of $W_{C_d}$ generated by the following simple reflections
$$S\setminus\{s_{0}, s_{\lambda_1},\ldots,s_{\lambda_1+\cdots+\lambda_{r-1}}\},$$
and set $x_{\lambda} :=\sum_{w\in W_{\lambda}}q^{\ell'(w)}T_{w}$, where $\ell'$ is the length function on $W_{C_d}$ (cf. \cite[\S1.1]{Lu03}).

Following \cite[\S5.1]{BKLW18} and also \cite[(6.1.2)]{LL21} we define the {\em $\imath$-Schur algebra} of type $C_d$ over $\A$ by
$$S^{\imath}(n, d)=\End_{\HH}(\bigoplus_{\lambda\in \Lambda^{\imath}(n, d)}x_{\lambda}\HH).$$

For $\lambda, \mu\in \Lambda^{\imath}(n, d)$, we denote by $\D_{\lambda\mu}$ (resp. $\D_{\lambda\mu}^{+}$) the set of minimal (resp. maximal) length double coset representatives in $W_\lambda\backslash W_{C_d}/W_\mu$. Set
\begin{equation*}   \label{eq:Xi}
\Xi^{\imath} :=\{(\lambda,g,\mu)~|~\lambda,\mu\in\Lambda^{\imath}(n, d), g\in\D_{\lambda\mu}\}
\end{equation*}
and
\begin{equation*}   \label{eq:Xi}
\widetilde{\Xi}^{\imath} :=\{(\lambda,g,\mu)~|~\lambda,\mu\in\Lambda^{\imath}(n, d), g\in\D_{\lambda\mu}^{+}\}.
\end{equation*}
Then there exists a bijection between $\Xi^{\imath}$ and $\widetilde{\Xi}^{\imath}$ (cf. \cite[Proposition 9.15]{Lu03}) and we shall identify them.

Let $D=2d$. Following \cite[(6.1.1)]{LL21} we set
$$\Theta_{n,d}^{\imath} :=\{(a_{ij})_{i, j\in [-r, r]\setminus \{0\}}\in \mathrm{Mat}_{n\times n}(\mathbb{N})\:|\:\sum_{i, j}a_{ij}=D\}$$
and
$$\Pi_{n, d}^{\imath} :=\{(a_{ij})\in \Theta_{n,d}^{\imath}\:|\:a_{ij}=a_{-i,-j}\mathrm{~for~all~}i,j\}.$$
By \cite[Lemma 6.1.1]{LL21}, there is a natural bijection between $\Xi^{\imath}$ or $\widetilde{\Xi}^{\imath}$ and $\Pi_{n, d}^{\imath}$, and we shall identify them. If $A\in \Pi_{n, d}^{\imath}$, we denote by $(ro(A), w_{A}, co(A))$ and $(ro(A), w_{A}^{+}, co(A))$ the corresponding element in $\Xi^{\imath}$ and $\widetilde{\Xi}^{\imath}$, respectively.

For $\lambda, \mu\in \Lambda^{\imath}(n, d)$ and $w\in \D_{\lambda\mu}$, we define $\phi_{\lambda, \mu}^{w}\in S^{\imath}(n, d)$ by
$$\phi_{\lambda, \mu}^{w}(x_{\nu}h)=\delta_{\mu, \nu}\sum_{z\in W_\lambda wW_\mu}q^{\ell'(z)}T_{z}h\quad \mathrm{for~}h\in \HH.$$
Then the set $\{\phi_{\lambda, \mu}^{w}\:|\:\lambda, \mu\in \Lambda^{\imath}(n, d)\mathrm{~and~}w\in \D_{\lambda\mu}\}$ forms an $\A$-basis of $S^{\imath}(n, d)$ (see \cite[\S6.1]{LL21}).

Using the above basis, similar to \eqref{g1}, we can define a standard basis of $S^{\imath}(n, d)$, denoted by $\{[A]\:|\:A\in \Pi_{n, d}^{\imath}\}$ (see loc. cit.). Moreover, we can define a bar involution $\bar{\cdot}$ on $S^{\imath}(n, d)$ and a partial order $<$ on $\Pi_{n, d}^{\imath}$, respectively (cf. \eqref{d2}-\eqref{d3}). Using these, in \cite[Theorem 5.9]{BKLW18}, a {\em canonical basis} for $S^{\imath}(n, d)$ has been constructed, which is denoted by $\{\{A\}\:|\:A\in \Pi_{n, d}^{\imath}\}$.



The following proposition provides a classification of left, right and two-sided cells for $S^{\imath}(n, d)$ with respect to the canonical basis $\{\{A\}\:|\:A\in \Pi_{n, d}^{\imath}\}$, which can be proved in a manner similar to Proposition \ref{lem:d=d2=d000ad}.
\begin{prop}
\label{lem:d=d2=d000ad-i-i}
For $A, B\in \Pi_{n, d}^{i}$, we have the following results.
\begin{enumerate}
\item
$\{A\}\preceq_{L}\{B\}$ if and only if $co(A)=co(B)$ and $w_{A}^{+}\preceq_{L} w_{B}^{+}$. Similarly, $\{A\}\preceq_{R}\{B\}$ if and only if $ro(A)=ro(B)$ and $w_{A}^{+}\preceq_{R} w_{B}^{+}$.
\item
$\{A\}\sim_{L}\{B\}$ if and only if $co(A)=co(B)$ and $w_{A}^{+}\sim_{L} w_{B}^{+}$. Similarly, $\{A\}\sim_{R}\{B\}$ if and only if $ro(A)=ro(B)$ and $w_{A}^{+}\sim_{R} w_{B}^{+}$.
\item
$\{A\}\sim_{LR}\{B\}$ if and only if $w_{A}^{+}\sim_{LR} w_{B}^{+}$.
\end{enumerate}
\end{prop}

\subsection{Another description of two-sided cells in $\HH$}

In this subsection, we shall recall another description of two-sided cells in $\HH$ with respect to the Kazhdan--Lusztig basis following \cite{BV82}.

We define the symbols in type $C_{d}$ and the set $\Phi_{d}$ as in \S\ref{d5}. Given a symbol
\begin{align*}
\Lambda={\lambda_1<\lambda_2<\cdots <\lambda_{m+1} \choose \mu_1<\mu_2<\cdots <\mu_{m}}
\end{align*}
for some $m\in \mathbb{Z}_{\geq 0}$, we take the set $\{2\lambda_{i}, 2\mu_{j}+1\:|\:1\leq i\leq m+1, 1\leq j\leq m\}$ and order it in a decreasing sequence, say $(\nu_{2m+1},\ldots,\nu_{1})$. Since $\sum_{i=1}^{2m+1}\nu_{i}=2d+2m^{2}+m$ and $\nu_{2m+1}>\nu_{2m}>\cdots>\nu_{1}$, we see that $(\nu_{2m+1}-(2m+1)+1, \nu_{2m}-2m+1,\ldots,\nu_1-1+1)$ is a partition of $2d$, which we shall denote by $par_{\imath}(\Lambda)$. By definition, we see that if $\Lambda\sim\Lambda'$, then $par_{\imath}(\Lambda)$ differs from $par_{\imath}(\Lambda')$ by possibly some 0's, which we regard as the same partition and denote it by $par_{\imath}[\Lambda]$. Let $\mathfrak{P}_{\imath}$ be the set of partitions $par_{\imath}[\Lambda]$, where $[\Lambda]$ runs over $\Phi_{d}$.

Similar to $\mathfrak{S}_{2d+1}$, we can define the Robinson--Schensted algorithm for $\mathfrak{S}_{2d}$. We shall identify $W_{C_d}$ with the group of permutations $w$ on the set $\{-d, -(d-1),\ldots,-1,1,$ $\ldots,d-1,d\}$ such that $w(-i)=-w(i)$ for any $i$. Under the identification, we have
\begin{align*}
s_{0}=(-1, 1),\quad s_{i}=(-i-1, -i)(i, i+1) \text{ for $1\leq i\leq d-1$.}
\end{align*}
For $w\in W_{C_d}$, we regard it as an element of $\mathfrak{S}_{2d}$ under the identification $-d\leftrightarrow 1, \ldots, -1\leftrightarrow d, 1\leftrightarrow d+1, \ldots, d\leftrightarrow 2d$; by applying the Robinson--Schensted correspondence for $\mathfrak{S}_{2d}$, we obtain a pair of standard tableaux $(T_{\imath}(w), T_{\imath}(w^{-1}))$ attached to it. We denote by $PT_{\imath}(w)$ the shape of $T_{\imath}(w)$, which is a partition of $2d$. Then we have the following proposition due to Barbasch and Vogan.
\begin{prop} $($see \cite[Proposition 17]{BV82}$)$
\label{lem:d=d2=d000adadc-i}
For any $w\in W_{C_d}$, $PT_{\imath}(w)$ belongs to the set $\mathfrak{P}_{\imath}$.
\end{prop}

Recall that in \S\ref{d5} we have defined an equivalence relation $\approx$ on $\Phi_{d}$. By an argument on \cite[p.80]{Mc96}, we see if $par_{\imath}[\Lambda]=par_{\imath}[\Lambda']$, then $[\Lambda]=[\Lambda']$. Therefore, the equivalence relation $\approx$ on $\Phi_{d}$ induces an equivalence relation on $\mathfrak{P}_{\imath}$, which we denote also by $\approx$; we say that $par_{\imath}[\Lambda]\approx par_{\imath}[\Lambda']$ if and only if $[\Lambda]\approx[\Lambda']$. Then we have another explicit description of two-sided cells in $\HH$.
\begin{prop} $($see \cite[Theorem 18]{BV82}$)$
\label{theor:d=d2=d000adadc-i}
For any two elements $w, w'\in W_{C_d}$, we have $w\sim_{LR} w'$ if and only if $PT_{\imath}(w)\approx PT_{\imath}(w')$.
\end{prop}


We denote by $\mathcal{P}_{d}^{\imath}$ the set of partitions of $2d$ in which every {\em odd} part appears with even multiplicity. Then there exists a bijection between $\mathcal{P}_{d}^{\imath}$ and the set of nilpotent orbits in type $C_{d}$ (cf. \cite[Theorem 5.1.3]{CoMc93}). For $\lambda\in \mathcal{P}_{d}^{\imath}$, we call it {\em special} if we also have $\lambda^{t}\in \mathcal{P}_{d}^{\imath}$ (cf. \cite[\S6.3]{CoMc93}). Denote by $\mathcal{SP}_{d}^{\imath}$ the set of special partitions of $2d$. Then the map, $\pi_{\imath} :[\Lambda]\mapsto par_{\imath}[\Lambda]$, is a bijection between $\mathcal{G}_{W_{B}}$ and $\mathcal{SP}_{d}^{\imath}$ (cf. \cite[p.80]{Mc96}). Thus, by Proposition \ref{theor:d=d2=d000adadc-i}, we see that there is a bijection between the set of two-sided cells in $\HH$ and $\mathcal{SP}_{d}^{\imath}$.


\subsection{Description of two-sided cells in $S^{\imath}(n, d)$}

In this subsection, we shall give a combinatorial description of two-sided cells in $S^{\imath}(n, d)$.

Set $([-r, r]\setminus \{0\})^{2}=([-r, r]\setminus \{0\})\times ([-r, r]\setminus \{0\})\subset \mathbb{Z}\times \mathbb{Z}$. It is easy to see that $([-r, r]\setminus \{0\})^{2}$ is a finite partially ordered set with the induced order by setting $(i, j)\leq (i', j')$ if $i\geq i'$ and $j\leq j'$. If $A=(a_{ij})_{i, j\in [-r, r]\setminus \{0\}}\in \Pi_{n, d}^{\imath}$ and $F$ is a $k$-chain family of $([-r, r]\setminus \{0\})^{2}$, we define $\mathfrak{s}_{F}(A)$ to be the sum of the entries $a_{ij}$ with $(i, j)\in F$, and call it the $F$-sum of $A$. Let $\mathfrak{s}_{k}(A)$ be the maximum value of $F$-sums of $A$ for all $k$-chain families $F$. We have the following result similar to Theorem \ref{theor:d=d2=d000=du-cells}.

\begin{thm} $($see \cite[Theorem 1.2]{Du96}$)$
\label{theor:d=d2=d000=du-cells-i-i}
For each $A\in \Pi_{n, d}^{\imath}$, we define $\sigma_{\imath}(A)=\mathfrak{s}_{i}(A)-\mathfrak{s}_{i-1}(A)$ (with the convention that $\mathfrak{s}_{0}(A)=0$) for all $1\leq i\leq n$. Then $\sigma_{\imath}(A)=(\sigma_{1}(A), \sigma_{2}(A), \ldots, \sigma_{n}(A))$ is a partition of $2d$.
\end{thm}
Recall that in \S\ref{d4} we have associated a permutation $y_{A}$ to each $A\in \Pi_{n, d}$. In a similar way, for each $A\in \Pi_{n, d}^{\imath}$, we can associate a permutation $\tilde{y}_{A}$ to it by replacing $\{-d, -(d-1),\ldots,-1,0,1,\ldots,d-1,d\}$ with $\{-d, -(d-1),\ldots,-1,1,\ldots,d-1,d\}$.


By \cite[Lemma 6.1.1]{LL21}, we have an analog of Lemma \ref{a3} for $W_{C_d}$ and $\Pi_{n, d}^{\imath}$. By arguments similar to the proofs of Lemmas \ref{longest element}, \ref{partition-longest element} and Theorem \ref{theor:charac of cells in Schur algs}, we can prove the following three results (using Propositions \ref{lem:d=d2=d000ad-i-i} and \ref{theor:d=d2=d000adadc-i}).
\begin{lem}\label{longest element-i}
Set $I^{\imath}=[1,r]\times ([-r, r]\setminus \{0\})\subset \mathbb{Z}\times \mathbb{Z}$. For each $A=(ro(A), w_{A}^{+}, co(A))\in \Pi_{n, d}^{\imath}$, we have $\tilde{y}_{A}=w_{A}^{+}$, and moreover,
\begin{align}\label{e1}
\ell'(w_{A}^{+})=d^{2}-\frac{1}{2}\sum_{(i,j)\in I^{\imath}}\Bigg(\sum\limits_{\substack{x< i\\y< j}}a_{ij}a_{xy}+\sum\limits_{\substack{x> i\\y> j}}a_{ij}a_{xy}\Bigg)-\frac{1}{2}\sum\limits_{\substack{i> 0\\j> 0}}a_{ij}.
\end{align}
\end{lem}

\begin{lem}
\label{partition-longest element-i}
For each $A=(ro(A), w_{A}^{+}, co(A))\in \Pi_{n, d}^{\imath}$, we have $\sigma_{\imath}(A)=PT_{\imath}(w_{A}^{+})$.
\end{lem}

\begin{thm}
\label{theor:charac of cells in Schur algs-i}
For any two elements $A, A'\in \Pi_{n, d}^{\imath}$, we have $\{A\}\sim_{LR} \{A'\}$ if and only if $\sigma_{\imath}(A)\approx \sigma_{\imath}(A')$.
\end{thm}

\begin{rem}\label{remark:onetoone correspondence-i}
Similar to the arguments in Remark \ref{remark:onetoone correspondence}, using Theorem \ref{theor:charac of cells in Schur algs-i} we can further show that there is a one-to-one correspondence between the set of two-sided cells in $S^{\imath}(n, d)$ and special partitions of $2d$ with at most $n$ parts.
\end{rem}

\begin{example}
(1) Assume that $n=2r=2$ and $d=2$. By \cite[Lemma 5.2]{BKLW18}, we have $\sharp \Pi_{2, 2}^{\imath}={2+2-1 \choose 2}=3$. All the $3$ elements in $\Pi_{2, 2}^{\imath}$ are as follows:
\begin{align*}
A_1=\left(\hspace{-1mm}
 \begin{array}{cc}
 2 & 0 \\
 0 & 2 \\
 \end{array}
\hspace{-1mm}\right),~~
A_2=\left(\hspace{-1mm}
 \begin{array}{ccc}
 1 & 1 \\
 1 & 1 \\
 \end{array}
\hspace{-1mm}\right),~~
A_3=\left(\hspace{-1mm}
 \begin{array}{ccc}
 0 & 2 \\
 2 & 0 \\
 \end{array}
\hspace{-1mm}\right).
\end{align*}
We have $\tilde{y}_{A_{1}}=s_{1}$, $\tilde{y}_{A_{2}}=s_{1}s_{0}s_{1}$, $\tilde{y}_{A_{3}}=s_{0}s_{1}s_{0}s_{1}$. By Lemma \cite[Lemma 6.1.1]{LL21}, we have $w_{A_{1}}=e$, $w_{A_{2}}=s_{0}$, $w_{A_{3}}=s_{0}s_{1}s_{0}$, and hence, $w_{A_{1}}^{+}=s_{1}$, $w_{A_{2}}^{+}=s_{1}s_{0}s_{1}$, $w_{A_{3}}^{+}=s_{0}s_{1}s_{0}s_{1}$, that is, $\tilde{y}_{A_{i}}=w_{A_{i}}^{+}$ for $1\leq i\leq 3$. By a direct calculation, we see that each $\ell'(w_{A_{i}}^{+})$ can be given by \eqref{e1}.

By Theorem \ref{theor:d=d2=d000=du-cells-i-i}, we have $\sigma_{\imath}(A_{1})=(2,2)$, $\sigma_{\imath}(A_{2})=(3,1)$, $\sigma_{\imath}(A_{3})=(4)$; using the Robinson--Schensted algorithm for $\mathfrak{S}_{4}$, we obtain $PT_{\imath}(w_{A_{1}}^{+})=(2,2)$, $PT_{\imath}(w_{A_{2}}^{+})=(3,1)$, $PT_{\imath}(w_{A_{3}}^{+})=(4)$. Therefore, $\sigma_{\imath}(A_{k})=PT_{\imath}(w_{A_{k}}^{+})$ for $1\leq k\leq 3$.

We have $(2,2)=par_{\imath}[\Lambda_{1}]$, $(3,1)=par_{\imath}[\Lambda_{2}]$, $(4)=par_{\imath}[\Lambda_{3}]$, where $\Lambda_{1}={0<2 \choose 1}$, $\Lambda_{2}={0<1 \choose 2}$, $\Lambda_{3}={0<3 \choose 0}$. Since $[\Lambda_{1}]\approx [\Lambda_{2}]\not\approx [\Lambda_{3}]$, that is, $(2,2)\approx (3,1)\not\approx (4)$, we see that $\{A_{1}\}, \{A_{2}\}$ lie in the same two-sided cell $\mathfrak{C}_1$, and $\{A_{3}\}$ lies in a different two-sided cell $\mathfrak{C}_2$. Since the special partitions of $4$ with at most $2$ parts are $(4)$ and $(2,2)$, we see that in this case, there is a bijection between the set of two-sided cells in $S^{\imath}(2, 2)$ and special partitions of $4$ with at most $2$ parts.

(2) Assume $n=2r=2$ and $d=3$. By \cite[Lemma 5.2]{BKLW18}, we have $\sharp \Pi_{2, 3}^{\imath}={2+3-1 \choose 3}=4$. All the $4$ elements in $\Pi_{2, 3}^{\imath}$ are as follows:
\begin{align*}
A_1=\left(\hspace{-1mm}
 \begin{array}{cc}
 3 & 0 \\
 0 & 3 \\
 \end{array}
\hspace{-1mm}\right),~~
A_2=\left(\hspace{-1mm}
 \begin{array}{ccc}
 2 & 1 \\
 1 & 2 \\
 \end{array}
\hspace{-1mm}\right),~~
A_3=\left(\hspace{-1mm}
 \begin{array}{ccc}
 1 & 2 \\
 2 & 1 \\
 \end{array}
\hspace{-1mm}\right),~~
A_4=\left(\hspace{-1mm}
 \begin{array}{ccc}
 0 & 3 \\
 3 & 0 \\
 \end{array}
\hspace{-1mm}\right).
\end{align*}

We have
\begin{align*}
\tilde{y}_{A_{1}}=s_{1}s_{2}s_{1},~ \tilde{y}_{A_{2}}=s_{2}s_{1}s_{0}s_{1}s_{2}s_{1},~ \tilde{y}_{A_{3}}=s_{1}s_{2}s_{0}s_{1}s_{0}s_{1}s_{2}s_{1}, ~\tilde{y}_{A_{4}}=s_{0}s_{1}s_{0}s_{2}s_{1}s_{0}s_{1}s_{2}s_{1}.
\end{align*}
By Lemma \cite[Lemma 6.1.1]{LL21}, we have $w_{A_{1}}=e$, $w_{A_{2}}=s_{0}$, $w_{A_{3}}=s_{0}s_{1}s_{0}$, $w_{A_{4}}=s_{0}s_{1}s_{0}s_{2}s_{1}s_{0}$, and hence, $\tilde{y}_{A_{i}}=w_{A_{i}}^{+}$ for $1\leq i\leq 4$. Moreover, each $\ell'(w_{A_{i}}^{+})$ can also be computed via \eqref{e1}.

By Theorem \ref{theor:d=d2=d000=du-cells-i-i}, we have $\sigma_{\imath}(A_{1})=(3,3)$, $\sigma_{\imath}(A_{2})=(4,2)$, $\sigma_{\imath}(A_{3})=(5,1)$, $\sigma_{\imath}(A_{4})=(6)$; using the Robinson--Schensted algorithm for $\mathfrak{S}_{6}$, we see that each $PT_{\imath}(w_{A_{k}}^{+})$ equals $\sigma_{\imath}(A_{k})$ for $1\leq k\leq 4$.

We have $(3,3)=par_{\imath}[\Lambda_{1}]$, $(4,2)=par_{\imath}[\Lambda_{2}]$, $(5,1)=par_{\imath}[\Lambda_{3}]$, $(6)=par_{\imath}[\Lambda_{4}]$, where $\Lambda_{1}={0<2 \choose 2}$, $\Lambda_{2}={0<3 \choose 1}$, $\Lambda_{3}={0<1 \choose 3}$, $\Lambda_{4}={0<4 \choose 0}$. Since $[\Lambda_{1}]\not\approx [\Lambda_{2}]\approx [\Lambda_{3}] \not\approx [\Lambda_{4}]$ and $[\Lambda_{1}]\not\approx [\Lambda_{4}]$, we see that there are exactly three two-sided cells in $S^{\imath}(2, 3)$: $\mathfrak{C}_1=\{\{A_{1}\}\}$, $\mathfrak{C}_2=\{\{A_{2}\}, \{A_{3}\}\}$, $\mathfrak{C}_3=\{\{A_{4}\}\}$. Since the special partitions of $6$ with at most $2$ parts are $(6)$, $(4,2)$ and $(3,3)$, we see that in this case, there is also a bijection between the set of two-sided cells in $S^{\imath}(2, 3)$ and special partitions of $6$ with at most $2$ parts.

(3) Assume $n=2r=2$ and $d$ is arbitrary. By \cite[Lemma 5.2]{BKLW18}, we have $\sharp \Pi_{2, d}^{\imath}={2+d-1 \choose d}=d+1$. All the $d+1$ elements in $\Pi_{2, d}^{\imath}$ are as follows:
\begin{align*}
\bigg\{A_k=\left(\hspace{-1mm}
 \begin{array}{cc}
 k & d-k \\
 d-k & k \\
 \end{array}
\hspace{-1mm}\right)\:\bigg|\:0\leq k\leq d\bigg\}.
\end{align*}
Moreover, $\sigma_{\imath}(A_{k})=(2d-k,k)$ for $0\leq k\leq d$.

It is easy to see that when $d$ is even, the special partitions of $2d$ with at most $2$ parts consist of the following $\frac{d}{2}+1$ partitions: $(2d), (2d-2, 2),\ldots, (d+2, d-2), (d,d)$; when $d$ is odd, the special partitions of $2d$ with at most $2$ parts consist of the following $\frac{d-1}{2}+2$ partitions: $(2d), (2d-2, 2), \ldots, (d+1, d-1), (d, d)$. Obviously, for each such special partition $\lambda$, there exists some $A_{k}$ such that $\sigma_{\imath}(A_{k})=\lambda$. Therefore, in this case there is a one-to-one correspondence between the set of two-sided cells in $S^{\imath}(2, d)$ and special partitions of $2d$ with at most $2$ parts.
\end{example}

Finally, similar to Conjecture \ref{theor:d=d2=d000adadcnumber}, we propose a conjecture on the number of left cells in a two-sided cell of $S^{\imath}(n, d)$. Let $\mathcal{SP}_{d}^{\imath, n}$ be the set of special partitions of $2d$ with at most $n$ parts. For each $\lambda\in \mathcal{SP}_{d}^{\imath, n}$, let $\mathbf{c}_{\lambda}^{\imath}$ be the associated two-sided cell in $S^{\imath}(n, d)$ by Remark \ref{remark:onetoone correspondence-i}.
\begin{conj}
\label{theor:d=d2=d000adadcnumber-i}
The number of left cells in $\mathbf{c}_{\lambda}^{\imath}$ equals that of semistandard domino tableaux of shape $\lambda$ with all entries in the dominoes $\leq r$.
\end{conj}

\section{An approach to determining two-sided cells in $\tilde{S}^{\imath}(n, d)$}
In this section, inspired by \cite{B17} we consider the $\tilde{\imath}$-Schur algebra $\tilde{S}^{\imath}(n, d)$ attached to $\mathcal{H}_{C_d}^{1}$, where $\mathcal{H}_{C_d}^{1}$ is the specialization at $p = 1$ of the Iwahori--Hecke algebra $\mathcal{H}_{C_d}^{p}$ of type $C_d$ with unequal parameters $p$ and $q$. We shall give an approach to determining whether or not two canonical basis elements of $\tilde{S}^{\imath}(n, d)$ lie in the same two-sided cell.

\subsection{Preliminaries}
  \label{sec:module}

Fix $d\in \Z_{\geq 2}$. Recall that $\A=\Z[q,q^{-1}]$. Let $p$ be another indeterminate and set $\mathcal{B}=\Z[p,p^{-1},q,q^{-1}]$. Let $\mathcal{H}_{C_d}^{p}$ be the Iwahori--Hecke algebra of type $C_d$ over $\mathcal{B}$ (cf. \cite[\S3.1]{B17}). It is generated by $T_{0}, T_1,\ldots, T_{d-1}$ with the following relations:
\begin{align*}\label{Btype-Hecke-alg-rela}
&T_{i}T_{i+1}T_{i}= T_{i+1}T_{i}T_{i+1}\quad \mathrm{for}~1\leq i\leq d-2,\\
&T_{0}T_{1}T_{0}T_{1}= T_{1}T_{0}T_{1}T_{0},\qquad T_{i}T_{j}=T_{j}T_{i}\quad \mathrm{if}~  |i-j|\geq 2,\\
&(T_{0}-p)(T_{0}+p^{-1})=0,\qquad (T_{i}-q)(T_i+q^{-1})=0\quad \mathrm{for}~ 1\leq i\leq d-1.
\end{align*}


Let $w\in W_{C_d}$. If $s_{i_1}s_{i_2}\cdots s_{i_r}$ is a {\em reduced expression} of it, we set $T_w :=T_{i_1}T_{i_2}\cdots T_{i_r}$. It is well-known that $T_w$ is independent of the choice of the reduced expression of $w$ (cf. \cite[\S3.2]{Lu03}). Let $\mathcal{H}_{C_d}^{1}$ denote the Iwahori--Hecke algebra of type $C_d$ over $\A$ with the parameter $p=1$.

Let $\mathcal{H}_{D_d}$ be the Iwahori--Hecke algebra of type $D_d$ over $\A$ (cf. \cite[\S10.2]{ES18} and also \cite[\S3.2]{B17}). It is generated by $T_{0}, T_1,\ldots, T_{d-1}$ with the following relations:
\begin{align*}
&T_{i}T_{i+1}T_{i}= T_{i+1}T_{i}T_{i+1}\quad \mathrm{for}~1\leq i\leq d-2,\qquad T_{0}T_{2}T_{0}= T_{2}T_{0}T_{2},\\
&T_{i}T_{j}=T_{j}T_{i}\quad \mathrm{if}~  1\leq i, j\leq d-1~\mathrm{ with}~|i-j|\geq 2,\qquad T_{0}T_{k}=T_{k}T_{0}\quad \mathrm{for}~k\neq 2,\\
&(T_{i}-q)(T_i+q^{-1})=0 \quad \mathrm{for}~ 0\leq i\leq d-1.
\end{align*}

Fix $n=2r\in \mathbb{Z}_{\geq 2}$. Recall that in \S\ref{e2}, we have defined the set $\Lambda^{\imath}(n, d)$, the length function $\ell'$ on $W_{C_d}$, and the parabolic subgroup $W_{\lambda}$ of $W_{C_d}$ associated to $\lambda\in \Lambda^{\imath}(n, d)$. Note that each $W_{\lambda}$ is generated by some simple reflections $s_{i}$'s ($1\leq i\leq d-1$).


We define a function $\ell'_{\mathfrak{a}}$ on $W_{C_d}$ by letting $\ell'_{\mathfrak{a}}(w)$ be the total number of $s_{i}$'s ($1\leq i\leq d-1$) in a reduced expression of $w$; in particular, we have $\ell'_{\mathfrak{a}}(s_{0})=0$ and $\ell'_{\mathfrak{a}}(s_{i})=1$ for $1\leq i\leq d-1$. Then $\ell'_{\mathfrak{a}}$ is a {\em weight function} (cf. \cite[\S3.1]{Lu03} and \cite[Example 2.4.4(a)]{GJ11}). Thus, $\mathcal{H}_{C_d}^{1}$ can be regarded as the Iwahori--Hecke algebra over $\A$ associated to the weight function $\ell'_{\mathfrak{a}}$, and we can apply the results in \cite{Lu03} and \cite[\S2.4]{GJ11}.


Following \cite[\S3.1]{B17} we define the {\em $\tilde{\imath}$-Schur algebra} $\tilde{S}^{\imath}(n, d)$ associated to $\mathcal{H}_{C_d}^{1}$ over $\A$ by
\begin{align*}
\tilde{S}^{\imath}(n, d)=\End_{\mathcal{H}_{C_d}^{1}}(\bigoplus_{\lambda\in \Lambda^{\imath}(n, d)}x_{\lambda}\mathcal{H}_{C_d}^{1}),\quad \text{ where }x_{\lambda}=\sum\limits_{w\in W_{\lambda}}q^{\ell'_{\mathfrak{a}}(w)}T_{w}.
\end{align*}

\begin{rem}\label{more-general-gene}
In \cite[\S2.2]{BWW18} and also \cite[\S6.1]{LL21}, the authors have given a more general construction of $\tilde{S}^{\imath}(n, d)$, which is associated to $\mathcal{H}_{C_d}^{p}$.
\end{rem}

For $\lambda, \mu\in \Lambda^{\imath}(n, d)$ and $g\in \D_{\lambda\mu}$, we define $\tilde{\phi}_{\lambda, \mu}^{g}\in \tilde{S}^{\imath}(n, d)$ by
\begin{align*}
\tilde{\phi}_{\lambda, \mu}^{g}(x_{\nu}h)=\delta_{\mu, \nu}T_{\lambda\mu}^{g} h\quad \mathrm{for~}h\in \mathcal{H}_{C_d}^{1},\text{ where }T_{\lambda\mu}^{g}=\sum_{w\in W_{\lambda}gW_{\mu}} q^{\ell'_{\mathfrak{a}}(w)}T_{w}.
\end{align*}
Then the set $\{\tilde{\phi}_{\lambda, \mu}^{g}\:|\:\lambda, \mu\in \Lambda^{\imath}(n, d)\mathrm{~and~}g\in \D_{\lambda\mu}\}$ forms an $\A$-basis of $\tilde{S}^{\imath}(n, d)$ (see \cite[\S6.1]{LL21}).

By \cite[Theorem 5.2]{Lu03}, the Kazhdan--Lusztig basis $\{\mathcal{C}'_{w}\:|\:w\in W_{C_d}\}$ of $\mathcal{H}_{C_d}^{1}$ can be defined. For $\lambda\in \Lambda^{\imath}(n, d)$, let $w^{\lambda}_{\circ}$ be the longest element in $W_{\lambda}$. By \cite[Proposition 1.17(ii)]{Xi94}, we have $\mathcal{C}'_{w^{\lambda}_{\circ}}=q^{-\ell'_{\mathfrak{a}}(w^{\lambda}_{\circ})}x_{\lambda}$. Let $<$ denote the Bruhat ordering on $W_{C_d}$. Since $\ell'_{\mathfrak{a}}(s_{i})=1$ for $1\leq i\leq d-1$, by \cite[Theorem 6.6(b) and Corollary 6.7(b)]{Lu03} we obtain the following lemma.
\begin{lem}\label{e3}
Given $1\leq i\leq d-1$ and $w\in W_{C_d}$, we have $T_{i}\mathcal{C}'_{w}=q \mathcal{C}'_{w}$ whenever $s_{i}w< w$; $\mathcal{C}'_{w}T_{i}=q \mathcal{C}'_{w}$ whenever $ws_{i}< w$.
\end{lem}

For $\lambda, \mu\in \Lambda^{\imath}(n, d)$, let $\mathcal{H}_{\lambda\mu}$ be the $\A$-submodule of $\mathcal{H}_{C_d}^{1}$ with a basis $\{T_{\lambda\mu}^{g}\}_{g\in \D_{\lambda\mu}}$. By Lemma \ref{e3}, we have $T_{w}x_{\lambda}=q^{\ell'_{\mathfrak{a}}(w)} x_{\lambda}$ for any $w\in W_{\lambda}$ (cf. also \cite[Lemma 3.1.1]{LL21}). By this and \cite[(4.1.4)]{LL21} (see \cite[(1.9)]{Cur85} or \cite[Lemma 7.33, Proposition 7.34]{DDPW08} for the equal parameter case), we have the following results.
\begin{lem}\label{f5}
We have
\begin{align*}
x_{\lambda}\mathcal{H}_{C_d}^{1}&=\big\{h\in \mathcal{H}_{C_d}^{1}\:\big|\:T_{w}h=q^{\ell'_{\mathfrak{a}}(w)}h ~\text{ for all } w\in W_{\lambda}\big\},\\
\mathcal{H}_{C_d}^{1}x_{\mu}&=\big\{h\in \mathcal{H}_{C_d}^{1}\:\big|\:hT_{w'}=q^{\ell'_{\mathfrak{a}}(w')}h ~\text{ for all } w'\in W_{\mu}\big\},\\
\mathcal{H}_{\lambda\mu}&=x_{\lambda}\mathcal{H}_{C_d}^{1}\cap \mathcal{H}_{C_d}^{1}x_{\mu}.
\end{align*}
\end{lem}
For $\lambda, \mu\in \Lambda^{\imath}(n, d)$ and $g\in \D_{\lambda\mu}$, let $g_{\lambda\mu}^{+}$ be the longest element in $W_{\lambda}gW_{\mu}$. By Lemmas \ref{e3} and \ref{f5}, similar to \cite[(1.10)]{Cur85} (cf. also \cite[Corollary 7.35]{DDPW08} and \cite[Lemma 3.8]{Du92}), we can prove the following results.
\begin{lem}\label{e4}
For $\lambda, \mu\in \Lambda^{\imath}(n, d)$, the set $\{\mathcal{C}'_{w}\:|\:w\in \D_{\lambda\mu}^{+}\}$ is an $\A$-basis of $\mathcal{H}_{\lambda\mu}$. Moreover, we have
\begin{equation}\label{g3}
\mathcal{C}'_{g_{\lambda\mu}^{+}}=q^{-\ell'_{\mathfrak{a}}(g_{\lambda\mu}^{+})}T_{\lambda\mu}^{g}+\sum\limits_{\substack{y\in \D_{\lambda\mu}\\y_{\lambda\mu}^{+} <g_{\lambda\mu}^{+}}} p_{y_{\lambda\mu}^{+}, g_{\lambda\mu}^{+}} \cdot q^{-\ell'_{\mathfrak{a}}(y_{\lambda\mu}^{+})}T_{\lambda\mu}^{y},
\end{equation}
where $p_{y_{\lambda\mu}^{+}, g_{\lambda\mu}^{+}}\in q^{-1}\Z[q^{-1}]$ for any $y_{\lambda\mu}^{+} <g_{\lambda\mu}^{+}$.
\end{lem}

We can define a bar involution $\bar{\cdot}$ on $\tilde{S}^{\imath}(n, d)$ similar to \eqref{d2}. We set $[\tilde{\phi}_{\lambda, \mu}^{g}]=q^{-\ell'_{\mathfrak{a}}(g_{\lambda\mu}^{+})+\ell'_{\mathfrak{a}}(w^{\mu}_{\circ})}\tilde{\phi}_{\lambda, \mu}^{g}$. By Lemma \ref{e4}, we have $\mathcal{C}'_{g^+_{\lambda\mu}}\in x_{\lambda}\mathcal{H}_{C_d}^{1}$. If we define
\begin{equation*}
\{\tilde{\phi}_{\lambda\mu}^g\} \in  \Hom_{\mathcal{H}_{C_d}^{1}}(x_{\mu}\mathcal{H}_{C_d}^{1}, x_{\lambda}\mathcal{H}_{C_d}^{1})\subset \tilde{S}^{\imath}(n, d)
\end{equation*}
by requiring
\begin{equation}\label{f1}
\{\tilde{\phi}_{\lambda\mu}^g\} (\mathcal{C}'_{w^{\mu}_\circ})=\mathcal{C}'_{g^+_{\lambda\mu}},
\end{equation}
then we have $\overline{\{\tilde{\phi}_{\lambda\mu}^g\}}=\{\tilde{\phi}_{\lambda\mu}^g\}$ (cf. \cite[Proposition 3.2(3)]{Du92}) and by \eqref{g3} (cf. \cite[(2.c)]{Du92}),
\begin{align*}
\{\tilde{\phi}_{\lambda\mu}^g\}=[\tilde{\phi}_{\lambda\mu}^g]+\sum\limits_{\substack{y\in \D_{\lambda\mu}\\y_{\lambda\mu}^{+} <g_{\lambda\mu}^{+}}} p_{y_{\lambda\mu}^{+}, g_{\lambda\mu}^{+}} [\tilde{\phi}_{\lambda\mu}^y].
\end{align*}
Thus, the set $\big\{\{\tilde{\phi}_{\lambda\mu}^g\}\:\big|\:\lambda, \mu\in \Lambda^{\imath}(n, d)\mathrm{~and~}g\in \D_{\lambda\mu}\big\}$ is an $\A$-basis of $\tilde{S}^{\imath}(n, d)$, which is called the canonical basis (cf. \cite[Theorem 6.2.3]{LL21}).

Recall that in \S\ref{e2}, we have a natural bijection between $\Xi^{\imath}$ or $\widetilde{\Xi}^{\imath}$ and $\Pi_{n, d}^{\imath}$, and for $A\in \Pi_{n, d}^{\imath}$, we denote by $(ro(A), w_{A}, co(A))$ and $(ro(A), w_{A}^{+}, co(A))$ the corresponding element in $\Xi^{\imath}$ and $\widetilde{\Xi}^{\imath}$, respectively. In the following, we set $\{A\}^{\heartsuit}=\{\tilde{\phi}_{\lambda\mu}^{w_{A}}\}$ if $A=(ro(A), w_{A}, co(A))=(\lambda, w_{A}, \mu)$.

Let $\tilde{h}_{x, y}^{z}$ (resp. $\tilde{g}_{A, B}^{C}$) denote the structure constants of $\mathcal{H}_{C_d}^{1}$ (resp. $\tilde{S}^{\imath}(n, d)$) with respect to the basis $\{\mathcal{C}'_{w}\:|\:w\in W_{C_d}\}$ (resp. $\{\{A\}^{\heartsuit}\:|\:A\in \Pi_{n, d}^{\imath}\}$), that is,
\[
\mathcal{C}'_{x}\cdot\mathcal{C}'_{y}=\sum_{z\in W_{C_d}}\tilde{h}_{x, y}^{z}\mathcal{C}'_{z},\qquad \{A\}^{\heartsuit}\cdot \{B\}^{\heartsuit}=\sum_{C\in \Pi_{n, d}^{\imath}}\tilde{g}_{A, B}^{C}\{C\}^{\heartsuit}.
\]

Similar to \cite[Lemma 5.1]{CLW20} and \cite[Proposition 3.4]{Du92}, using \eqref{f1} and Lemma \ref{e3} we can prove the following results.
\begin{lem}
\label{f2}
For any $A=(\lambda, w_A, \mu), B=(\mu, w_B, \nu), C=(\lambda, w_C, \nu)\in \Pi_{n, d}^{\imath}$, we have
\[
\pi(J_{\mu}) \cdot\tilde{g}_{A, B}^{C}=\tilde{h}_{w_{A}^{+}, w_{B}^{+}}^{w_{C}^{+}}, \text{ where } \pi(J_{\mu})=q^{-\ell'_{\mathfrak{a}}(w^{\mu}_{\circ})}\sum\limits_{w\in W_{\mu}}q^{2\ell'_{\mathfrak{a}}(w)}.
\]
\end{lem}

By applying the arguments in \cite[\S2.4.8]{GJ11} to $\mathcal{H}_{C_d}^{1}$, we see that
\begin{align}\label{f4}
\tilde{h}_{x, y}^{z}\in \mathbb{N}[q, q^{-1}]\text{ for all }x, y, z\in W_{C_d},
\end{align}
and (P1)--(P15) in \cite[Conj.~ 14.2]{Lu03} hold for $\mathcal{H}_{C_d}^{1}$. These imply that we can apply the results in \cite[\S18]{Lu03}; in particular, we can define the asymptotic algebra for $\mathcal{H}_{C_d}^{1}$ and establish an analog of \cite[Proposition 18.4]{Lu03} for it. We shall cite the results in loc. cit. directly.

Recall that $S=\{s_{0}, s_{1}, \ldots, s_{d-1}\}$. We set $\tilde{S}=\{s_{1}, \ldots, s_{d-1}\}$. Recall that $\ell'_{\mathfrak{a}}(t)=1$ for each $t\in \tilde{S}$. For $w\in W_{C_d}$, we set $\mathcal{L}(w)=\{s\in \tilde{S}\:|\:sw< w\}$ and $\mathcal{R}(w)=\{s\in \tilde{S}\:|\:ws< w\}$. In the remainder of this section, we shall write $\mathcal{C}'_{y}\preceq_{\star}\mathcal{C}'_{w}$ as $y\preceq_{\star}w$ and $\mathcal{C}'_{y}\sim_{\star}\mathcal{C}'_{w}$ as $y\sim_{\star}w$ for $\star\in \{L, R, LR\}$. Similar to \cite[Lemma 8.6]{Lu03}, we have the following lemma.
\begin{lem}\label{f3}
Let $w, w'\in W_{C_d}$.
\begin{itemize}
\item[(a)] If $w\preceq_{L} w'$, then $\mathcal{R}(w')\subseteq \mathcal{R}(w)$. If $w\sim_{L} w'$, then $\mathcal{R}(w')=\mathcal{R}(w)$.
\item[(b)] If $w\preceq_{R} w'$, then $\mathcal{L}(w')\subseteq \mathcal{L}(w)$. If $w\sim_{R} w'$, then $\mathcal{L}(w')=\mathcal{L}(w)$.
\end{itemize}
\end{lem}
\begin{proof}
To prove the first assertion of (a), it suffices to prove the case when $\mathcal{C}'_{w}\leftarrow_{L} \mathcal{C}'_{w'}$. We assume that the coefficient of $\mathcal{C}'_{w}$ is nonzero when expanding $\mathcal{C}'_{z}\mathcal{C}'_{w'}$ for some $z\in W_{C_d}$. Let $t\in \mathcal{R}(w')$. Then we have $\mathcal{C}'_{w'}\in {}^{t}\!\hspace{0.5mm}\mathcal{H}_{C_d}^{1}$, where ${}^{t}\!\hspace{0.5mm}\mathcal{H}_{C_d}^{1}=\oplus_{y;yt< y}\A\mathcal{C}'_{y}$. By\cite[Lemma 8.4(b)]{Lu03}, ${}^{t}\!\hspace{0.5mm}\mathcal{H}_{C_d}^{1}$ is a left ideal of $\mathcal{H}_{C_d}^{1}$. Hence $\mathcal{C}'_{z}\mathcal{C}'_{w'}\in {}^{t}\!\hspace{0.5mm}\mathcal{H}_{C_d}^{1}$. From the definition of ${}^{t}\!\hspace{0.5mm}\mathcal{H}_{C_d}^{1}$, we must have $wt< w$, that is, $t\in \mathcal{R}(w)$. Hence $\mathcal{R}(w')\subseteq \mathcal{R}(w)$. The second assertion of (a) follows immediately from the first one. The proof of (b) is entirely similar to that of (a).
\end{proof}

Now we can provide a classification of left, right and two-sided cells for $\tilde{S}^{\imath}(n, d)$ with respect to the canonical basis $\{\{A\}^{\heartsuit}\:|\:A\in \Pi_{n, d}^{\imath}\}$ (cf. \cite[Lemma 2.2 and Corollary 2.3]{Du96}).
\begin{prop}
\label{lem:d=d2=d000ad-i}
For $A, B\in \Pi_{n, d}^{\imath}$, we have the following results.
\begin{enumerate}
\item
$\{A\}^{\heartsuit}\preceq_{L}\{B\}^{\heartsuit}$ if and only if $co(A)=co(B)$ and $w_{A}^{+}\preceq_{L} w_{B}^{+}$. Similarly, $\{A\}^{\heartsuit}\preceq_{R}\{B\}^{\heartsuit}$ if and only if $ro(A)=ro(B)$ and $w_{A}^{+}\preceq_{R} w_{B}^{+}$.
\item
$\{A\}^{\heartsuit}\sim_{L}\{B\}^{\heartsuit}$ if and only if $co(A)=co(B)$ and $w_{A}^{+}\sim_{L} w_{B}^{+}$. Similarly, $\{A\}^{\heartsuit}\sim_{R}\{B\}^{\heartsuit}$ if and only if $ro(A)=ro(B)$ and $w_{A}^{+}\sim_{R} w_{B}^{+}$.
\item
$\{A\}^{\heartsuit}\sim_{LR}\{B\}^{\heartsuit}$ if and only if $w_{A}^{+}\sim_{LR} w_{B}^{+}$.
\end{enumerate}
\end{prop}
\begin{proof}
(1) The ``only if'' part follows from Lemma \ref{f2}.

Conversely, suppose that $co(A)=co(B)=\gamma$ and $w_{A}^{+}\preceq_{L} w_{B}^{+}$. By \eqref{f4} and Lemma \ref{another-charact-cells}, $\mathcal{C}'_{w_{A}^{+}}$ appears with nonzero coefficient in the product $\mathcal{C}'_{w}\mathcal{C}'_{w_{B}^{+}}$ for some $w$. Denote $A=(\lambda,w_{A},\gamma)$ and $B=(\mu,w_{B},\gamma)$ for $\lambda,\mu\in \Lambda^{\imath}(n, d)$. By \eqref{f4} and Lemma \ref{e3}, $\mathcal{C}'_{w_{A}^{+}}$ appears with nonzero coefficient in $\mathcal{C}'_{w_\circ^\lambda}\mathcal{C}'_{w}\mathcal{C}'_{w_\circ^\mu}\mathcal{C}'_{w_{B}^{+}}$. By Lemmas \ref{f5} and \ref{e4}, $\mathcal{C}'_{w_\circ^\lambda}\mathcal{C}'_{w}\mathcal{C}'_{w_\circ^\mu}$ is a linear combination of the elements $\mathcal{C}'_{w_{D}^{+}}$ ($D=(\lambda,w_D,\mu)\in \Pi_{n, d}^{\imath}$). Thus, $\mathcal{C}'_{w_{A}^{+}}$ appears with nonzero coefficient in some product $\mathcal{C}'_{w_{D}^{+}}\mathcal{C}'_{w_{B}^{+}}$. It follows from Lemma \ref{f2} that $\{A\}^{\heartsuit}$ appears with nonzero coefficient in $\{D\}^{\heartsuit}\{B\}^{\heartsuit}$, hence $\{A\}^{\heartsuit}\preceq_{L}\{B\}^{\heartsuit}$. The proof for the second claim on $\preceq_{R}$ is entirely similar.

(2) It follows from (1).

(3) The ``only if'' part follows from Lemma \ref{f2}.

Conversely, suppose that $w_{A}^{+}\sim_{LR} w_{B}^{+}$. By \cite[Proposition 18.4]{Lu03} there exists $x\in W_{C_d}$ such that $w_{A}^{+}\sim_{L} x \sim_{R}w_{B}^{+}$. Denote $co(A)=\gamma$ and $ro(B) =\mu$. By Lemma \ref{f3} and \cite[(1.2)(i)]{Cur85}, we see that $x$ is the longest element in $W_{\mu}x W_{\gamma}$, and hence $x=w_{D}^{+}$ for $D=(\mu,w_D,\gamma)\in \Pi_{n, d}^{\imath}$. By (2) we have $\{A\}^{\heartsuit}\sim_{L}\{D\}^{\heartsuit}\sim_{R}\{B\}^{\heartsuit}$, hence $\{A\}^{\heartsuit}\sim_{LR}\{B\}^{\heartsuit}$.
\end{proof}

\subsection{Description of two-sided cells in $\mathcal{H}_{D_d}$}
We define a {\em symbol} in type $D_{d}$ to be an array of nonnegative integers
\begin{align}\label{g2}
\Lambda={\lambda_1<\lambda_2<\cdots <\lambda_{m} \choose \mu_1<\mu_2<\cdots <\mu_{m}}
\end{align}
such that $\sum_{i=1}^{m}\lambda_{i}+\sum_{j=1}^{m}\mu_{j}=d+m(m-1)$, where $m\in \mathbb{Z}_{\geq 0}$ (cf. \cite[\S5]{Lu79}). We define an equivalence relation $\sim$ on the set of symbols in type $D_{d}$ as
\begin{equation*}  \label{dim:Sirrep-2}
{\lambda_1<\lambda_2<\cdots <\lambda_{m} \choose \mu_1<\mu_2<\cdots <\mu_{m}}\sim {0<\lambda_1+1<\lambda_2+1<\cdots <\lambda_{m}+1 \choose 0<\mu_1+1<\mu_2+1<\cdots <\mu_{m}+1},
\end{equation*}
and
\begin{equation*}  \label{dim:Sirrep-2}
{\lambda_1<\lambda_2<\cdots <\lambda_{m} \choose \mu_1<\mu_2<\cdots <\mu_{m}}\sim  {\mu_1<\mu_2<\cdots <\mu_{m} \choose \lambda_1<\lambda_2<\cdots <\lambda_{m}}.
\end{equation*}
Denote by $[\Lambda]$ the equivalence class of a symbol $\Lambda$ and $\tilde{\Phi}_{d}$ the set of equivalence classes of symbols relative to $\sim$. We make the convention that each symbol in \eqref{g2} with $\lambda_i=\mu_i$ for all $1\leq i\leq m$ should be counted {\em twice}, i.e., it gives rise to two elements of $\tilde{\Phi}_{d}$.

Let $W_{D_d}$ be the Coxeter group of type $D_d$ with simple reflections $s_{0}^{d}$ and $s'_j$ ($1\leq j\leq d-1$). It is known that $W_{D_d}$ can be embedded into $W_{C_d}$ via $s_{0}^{d}\mapsto s_0s_1s_0$ and $s'_j\mapsto s_j$ for $1\leq j\leq d-1$, recalling that $s_i$ ($0\leq i\leq d-1$) are the simple reflections in $W_{C_d}$ (cf. \cite[Example 2.4.4(a)]{GJ11}). We shall identify $W_{D_d}$ with its image in $W_{C_d}$ under the embedding. Thus, $W_{D_d}$ consists of the permutations
\begin{equation*}
\left(\hspace{-1.6mm}
\begin{array}{cccccc}
-d & \cdots & -1&1&\cdots& d\\
-i_d & \cdots & -i_1&i_1&\cdots& i_d
\end{array}
\hspace{-1.5mm}\right)
\end{equation*}
in $W_{C_d}$ such that the number of negative numbers among $i_1,\ldots,i_d$ is even.

Let $\alpha=(\alpha_{m}, \ldots, \alpha_2, \alpha_1)$, $\beta=(\beta_m, \ldots, \beta_2, \beta_1)$ be two partitions such that $\sum_{i=1}^{m}\alpha_{i}+\sum_{j=1}^{m}\beta_{j}=d, \alpha_1\geq 0, \beta_{1}\geq 0$. Let $\alpha'=(\alpha_{m}, \ldots, \alpha_2, \alpha_1,0)$. Let $E_{\alpha, \beta}$ be the representation of $W_{D_{d}}$ obtained by restriction of the representation $E_{\alpha', \beta}$ of $W_{C_d}$, where $E_{\alpha', \beta}$ is the irreducible representation of $W_{C_d}$ associated to the ordered pair of partitions $(\alpha', \beta)$. Then $E_{\alpha, \beta}=E_{\beta, \alpha}$ is irreducible if $\alpha\neq \beta$. If $\alpha=\beta$, $E_{\alpha, \alpha}$ splits into two distinct irreducible $W_{D_{d}}$-modules $E_{\alpha, \alpha}^{I}, E_{\alpha, \alpha}^{II}$. All irreducible representations of $W_{D_{d}}$ are obtained in this way. If $(\alpha, \beta)$ is a pair of partitions as above, with $\alpha\neq \beta$, we define $[\Lambda]\in \tilde{\Phi}_{d}$ by setting $\lambda_i :=\alpha_i+i-1$ $(1\leq i\leq m)$, $\mu_j :=\beta_j+j-1$ $(1\leq j\leq m)$. We then set $E^{\Lambda} :=E_{\alpha, \beta}$. If $\alpha=\beta$, the same formulae define two elements $[\Lambda(I)], [\Lambda(II)]$ of $\tilde{\Phi}_{d}$, and we set $E^{\Lambda(I)} :=E_{\alpha, \alpha}^{I}$, $E^{\Lambda(II)} :=E_{\alpha, \alpha}^{II}$. Thus, we see that there is a one-to-one correspondence between $\tilde{\Phi}_{d}$ and the irreducible representation of $W_{D_{d}}$ $($up to isomorphism$)$ (cf. \cite[\S5]{Lu79}).

By \cite[p.175]{BV82}, we can define a map $\chi$ from the set of symbols in type $C_d$ to that in type $D_d$ by
\begin{equation*}
\chi(\Lambda)={\lambda_1<\lambda_2<\cdots <\lambda_{m+1} \choose 0<\mu_1+1<\cdots <\mu_{m}+1}\text{ for }\Lambda={\lambda_1<\lambda_2<\cdots <\lambda_{m+1} \choose \mu_1<\mu_2<\cdots <\mu_{m}}.
\end{equation*}
For each $w\in W_{D_d}\subset W_{C_d}$, by the Robinson--Schensted algorithm for $\mathfrak{S}_{2d}$ and Proposition \ref{lem:d=d2=d000adadc-i}, we have $PT_{\imath}(w)=par_{\imath}[\Lambda]$ for a unique $[\Lambda]$, where $\Lambda$ is a symbol in type $C_d$. By definition, we have $\chi(\Lambda)\sim\chi(\Lambda')$ as symbols in type $D_d$ if $\Lambda\sim\Lambda'$, where $\Lambda, \Lambda'$ are two symbols in type $C_d$. Therefore, $w$ determines an equivalence class $[\chi(\Lambda)]$ of symbols in type $D_d$. We call $[\chi(\Lambda)]$ the {\em symbol class} attached to $w$, and denote it by $sym(w)$.



We now define an equivalence relation $\approx$ on $\tilde{\Phi}_{d}$. We say
$$\bigg[{\lambda_1<\lambda_2<\cdots <\lambda_{m} \choose \mu_1<\mu_2<\cdots <\mu_{m}}\bigg] \approx  \bigg[{\lambda_1'<\lambda_2'<\cdots <\lambda_{m}' \choose \mu_1'<\mu_2'<\cdots <\mu_{m}'}\bigg]$$
if and only if $$\{\lambda_1, \lambda_2, \ldots, \lambda_{m}, \mu_1, \mu_2, \ldots, \mu_{m}\}=\{\lambda_1', \lambda_2', \ldots, \lambda_{m}', \mu_1', \mu_2', \ldots, \mu_{m}'\}\mbox{ as two sets}.$$
We make the convention that the two elements in $\tilde{\Phi}_{d}$, which arise from a symbol in \eqref{g2} with $\lambda_i=\mu_i$ for all $i$, are not equivalent.

We shall identify $\mathcal{H}_{D_d}$ with a subalgebra of $\mathcal{H}_{C_d}^{1}$, and $\{\mathcal{C}'_{w}\:|\:w\in W_{D_d}\}$ is its Kazhdan--Lusztig basis (cf. \cite[Proposition 2.4.5]{GJ11}). The following proposition gives an explicit description of two-sided cells in $\mathcal{H}_{D_d}$ with respect to the basis.
\begin{prop} $($see \cite[Theorem 18]{BV82}$)$
\label{theor:d=d2=d000adadc-i-20190919}
For any two elements $w, w'\in W_{D_d}$, we have $w\sim_{LR} w'$ $($in $\mathcal{H}_{D_d}$$)$ if and only if $sym(w)\approx sym(w')$.
\end{prop}

Let $\mathcal{G}_{W_{D}}$ be the subset of $\tilde{\Phi}_{d}$ consisting of equivalence classes of symbols such that $\lambda_{i}\leq \mu_{i}\leq\lambda_{i+1}$ or $\mu_{i}\leq \lambda_{i}\leq\mu_{i+1}$ for any $i$. Obviously, each equivalence class of $\tilde{\Phi}_{d}$ relative to $\approx$ contains exactly one element of $\mathcal{G}_{W_{D}}$. We denote by $\tilde{\mathcal{P}}_{d}$ the set of partitions of $2d$ with every {\em even} part occurring an even number of times, where we shall count a {\em very even} partition (in which all parts are even) twice and use Roman numerals $I$ and $II$ to label them (cf. \cite[Theorem 5.1.4]{CoMc93}). An element $\lambda\in \tilde{\mathcal{P}}_{d}$ is called {\em special} if $\lambda^{t}\in \mathcal{P}_{d}^{\imath}$; in particular, all very even partitions are special (cf. \cite[\S6.3]{CoMc93}). Denote by $\widetilde{\mathcal{SP}}_{d}$ the set of special partitions in $\tilde{\mathcal{P}}_{d}$. Then there is a bijection between $\mathcal{G}_{W_{D}}$ and $\widetilde{\mathcal{SP}}_{d}$ (cf. \cite[p.80]{Mc96}). By Proposition \ref{theor:d=d2=d000adadc-i-20190919}, we see that there is a bijection between the set of two-sided cells in $\mathcal{H}_{D_d}$ and $\widetilde{\mathcal{SP}}_{d}$.

\subsection{Description of two-sided cells in $\tilde{S}^{\imath}(n, d)$}
We first recall a characterization of two-sided cells in $\mathcal{H}_{C_d}^{1}$. Recall that we have identified $W_{D_d}$ with a subgroup of $W_{C_d}$. We set $\Omega=\{e, s_{0}\}$. Then we have a semidirect product decomposition $W_{C_d}=\Omega\ltimes W_{D_d}$ (cf. \cite[\S2.4.3(a)]{GJ11}). The following proposition gives a characterization of two-sided cells in $\mathcal{H}_{C_d}^{1}$ in terms of those in $\mathcal{H}_{D_d}$.
\begin{prop} $($see \cite[Proposition 2.4.9]{GJ11}$)$
\label{theor:characterization of two-sided cells in HCd1}
Let $\omega_1, \omega_2\in \Omega$ and $x_1, x_2\in W_{D_d}$. Then $\omega_1x_1\preceq_{LR} \omega_2x_2$ $($in $\mathcal{H}_{C_d}^{1}$$)$ if and only if there exists some $\omega\in \Omega$ such that $x_1 \preceq_{LR} \omega x_2\omega^{-1}$ $($in $\mathcal{H}_{D_d}$$)$. Thus, the two-sided cells in $\mathcal{H}_{C_d}^{1}$ are of the form $\Omega\cdot \mathcal{F}\cdot\Omega$, where $\mathcal{F}$ is a two-sided cell in $\mathcal{H}_{D_d}$.
\end{prop}

There is an outer automorphism on $W_{D_d}$ induced by the simple reflection $s_0$, which fixes the two-sided cells corresponding to special partitions which have odd parts, and permutes the very even partitions labelled by $I$ and $II$.

Let $\hat{\mathcal{P}}_{d}$ denote the set of partitions of $2d$ in which the multiplicity of every {\em even} part is even. An element $\lambda\in \hat{\mathcal{P}}_{d}$ is {\em special} if $\lambda^{t}\in \mathcal{P}_{d}^{\imath}$. Let $\widehat{\mathcal{SP}}_{d}$ be the set of special partitions in $\hat{\mathcal{P}}_{d}$. By Proposition \ref{theor:characterization of two-sided cells in HCd1}, we see that there is a one-to-one correspondence between the set of two-sided cells in $\mathcal{H}_{C_d}^{1}$ and $\widehat{\mathcal{SP}}_{d}$.

Recall that for each $A\in \Pi_{n, d}^{\imath}$, in Lemma \ref{longest element-i} we have constructed the element $w_{A}^{+}\in W_{C_d}$ attached to it. Thus, by Propositions \ref{lem:d=d2=d000ad-i}-\ref{theor:characterization of two-sided cells in HCd1}, we can give an approach to determining whether or not two canonical basis elements of $\tilde{S}^{\imath}(n, d)$ lie in the same two-sided cell.

We propose the following conjecture.
\begin{conj}
Assume that $w\in W_{D_d}\subset W_{C_d}$. Then we have $sym(s_0w)=sym(w)$, where the two elements in $\tilde{\Phi}_{d}$, which arise from a symbol in \eqref{g2} with $\lambda_i=\mu_i$ for all $i$, will be regarded as identical, and $sym(s_0w)$ is defined in a manner similar to $sym(w)$.
\end{conj}
In fact, if the conjecture is proved, similar to Theorem \ref{theor:charac of cells in Schur algs-i} we can obtain a combinatorial description of two-sided cells in $\tilde{S}^{\imath}(n, d)$.

We illustrate the above conjecture with some examples.
\begin{example}
Assume that $d=2$. \\
(1) We have $PT_{\imath}(s_0)=(2,1,1)=par_{\imath}[{1<2 \choose 0}]$ and $PT_{\imath}(e)=(1,1,1,1)=par_{\imath}[{0<1<2 \choose 1<2}]$. Then we have $sym(s_0)=[{1<2 \choose 0<1}]=[{0<1 \choose 1<2}]=[{0<1<2 \choose 0<2<3}]=sym(e)$.\\
(2) We have $PT_{\imath}(s_0s_1)=(2,2)=par_{\imath}[{0<2 \choose 1}]$ and $PT_{\imath}(s_1)=(2,2)=par_{\imath}[{0<2 \choose 1}]$. Then we have $sym(s_0s_1)=[{0<2 \choose 0<2}]=sym(s_1)$.\\
(3) We have $PT_{\imath}(s_0s_1s_0)=(2,2)=par_{\imath}[{0<2 \choose 1}]$ and $PT_{\imath}(s_1s_0)=(2,2)=par_{\imath}[{0<2 \choose 1}]$. Then we have $sym(s_0s_1s_0)=[{0<2 \choose 0<2}]=sym(s_1s_0)$.\\
(4) We have $PT_{\imath}(s_0s_1s_0s_1)=(4)=par_{\imath}[{0<3 \choose 0}]$ and $PT_{\imath}(s_1s_0s_1)=(3,1)=par_{\imath}[{0<1 \choose 2}]$. Then we have $sym(s_0s_1s_0s_1)=[{0<3 \choose 0<1}]=[{0<1 \choose 0<3}]=sym(s_1s_0s_1)$.

Assume that $d=3$. We have $PT_{\imath}(s_0s_1s_0s_2s_1s_0s_1s_2s_1)=(6)=par_{\imath}[{0<4 \choose 0}]$ and $PT_{\imath}(s_1s_0s_2s_1s_0s_1s_2s_1)=(5,1)=par_{\imath}[{0<1 \choose 3}]$. Then we have $sym(s_0s_1s_0s_2s_1s_0s_1s_2s_1)=[{0<4 \choose 0<1}]=[{0<1 \choose 0<4}]=sym(s_1s_0s_2s_1s_0s_1s_2s_1)$.

Assume that $d=4$. The 192 elements in $W_{D_4}$ are listed on \cite[pp.176-178]{BV82}. We leave it to the reader to check the cases.
\end{example}

\begin{rem}\label{remark-tilde-Schur}
Since the $\tilde{\imath}$-Schur algebra $\tilde{S}^{\imath}(n, d)$ is a specialization of a more general construction in \cite[\S6.1]{LL21} (cf. Remark \ref{more-general-gene}), we can apply their results to our situation.

In \cite[\S6.3]{LL21}, a stabilization algebra $\dot{\mathbb{K}}_{n}^{\imath}$ has been constructed from the family of $\tilde{\imath}$-Schur algebras $\tilde{S}^{\imath}(n, d)$ as $d$ varies. From its construction, it should be possible to define a surjective algebra homomorphism from $\dot{\mathbb{K}}_{n}^{\imath}$ to $\tilde{S}^{\imath}(n, d)$ for each $d$ (cf. \cite[Proposition A.17]{BKLW18}). Assume that $\mathbf{U}_{q}^{\imath}(\mathfrak{g}\mathfrak{l}_{n})$ and $\mathbf{U}_{q}^{\imath}(\mathfrak{s}\mathfrak{l}_{n})$ are the $\imath$quantum groups associated to $\mathfrak{g}\mathfrak{l}_{n}$ and $\mathfrak{s}\mathfrak{l}_{n}$, which are the coideal subalgebras of $\mathbf{U}_{q}(\mathfrak{g}\mathfrak{l}_{n})$ and $\mathbf{U}_{q}(\mathfrak{s}\mathfrak{l}_{n})$, respectively; let $\dot{\mathbf{U}}_{q}^{\imath}(\mathfrak{g}\mathfrak{l}_{n})$ and $\dot{\mathbf{U}}_{q}^{\imath}(\mathfrak{s}\mathfrak{l}_{n})$ are the modified $\imath$quantum groups attached to $\mathbf{U}_{q}^{\imath}(\mathfrak{g}\mathfrak{l}_{n})$ and $\mathbf{U}_{q}^{\imath}(\mathfrak{s}\mathfrak{l}_{n})$; we refer the reader to \cite[\S6.4]{LiW18} and also \cite[\S7.3]{LL21} for the definition of these algebras.

In \cite[Theorem 7.3.2]{LL21}, an algebra isomorphism from $\dot{\mathbf{U}}_{q}^{\imath}(\mathfrak{g}\mathfrak{l}_{n})$ to $\dot{\mathbb{K}}_{n}^{\imath}$ has been established, whose composition with the possible homomorphism from $\dot{\mathbb{K}}_{n}^{\imath}$ to $\tilde{S}^{\imath}(n, d)$ will yield a surjective algebra homomorphism from $\dot{\mathbf{U}}_{q}^{\imath}(\mathfrak{g}\mathfrak{l}_{n})$ to $\tilde{S}^{\imath}(n, d)$. Then, similar to the construction of $\phi_{d}^{\imath}$ in \cite[(6.11)]{LiW18}, we can construct an algebra homomorphism $\tilde{\phi}_{d}^{\imath} :\dot{\mathbf{U}}_{q}^{\imath}(\mathfrak{s}\mathfrak{l}_{n})\rightarrow \tilde{S}^{\imath}(n, d)$.

We conjecture that $\tilde{\phi}_{d}^{\imath}$ sends each canonical basis element of $\dot{\mathbf{U}}_{q}^{\imath}(\mathfrak{s}\mathfrak{l}_{n})$ to a canonical basis element of $\tilde{S}^{\imath}(n, d)$ or zero, and moreover, its kernel is spanned by the canonical basis elements whose images are zero under $\tilde{\phi}_{d}^{\imath}$. If the conjecture holds, we can lift the combinatorial description of two-sided cells in $\tilde{S}^{\imath}(n, d)$ to give a characterization of two-sided cells in $\dot{\mathbf{U}}_{q}^{\imath}(\mathfrak{s}\mathfrak{l}_{n})$, which is similar to the claim of Theorem \ref{characte-two-sided-cells-j} in Section 4.
\end{rem}




\end{document}